\documentclass[11pt,reqno,twoside]{amsart}

\usepackage{amsmath}
\usepackage{amsfonts}
\usepackage{amssymb}

\usepackage{epsfig}

\usepackage{xcolor}

\numberwithin{equation}{section}

\newif\ifdraft\drafttrue

\draftfalse

\newcommand\mr{M_{m,n}}

\newcommand\hd{Hausdorff dimension}

\newcommand\da{Diophantine approximation}

\newcommand\ssm{\smallsetminus}

\newcommand\eq[2]{{\ifdraft{\ \tt [#1]}\else\ignorespaces\fi}\begin{equation}\label{eq:#1}{#2}\end{equation}}
\newcommand {\equ}[1]     {\eqref{eq:#1}}

\newcommand{\R}{{\mathbb {R}}}
\newcommand{\Z}{{\mathbb {Z}}}
\newcommand{\N}{{\mathbb {N}}}

\newcommand{\Leb}{{\operatorname{Leb}}}

\newcommand{\Ad}{{\operatorname{Ad}}}

\newcommand{\SL}{\operatorname{SL}}

\newcommand{\ggm}{G/\Gamma}

\newcommand{\dist}{\operatorname{dist}}
\newcommand{\diag}{\operatorname{diag}}

\newcommand{\Lie}{\operatorname{Lie}}

\newcommand {\ignore}[1]  {}

\newcommand\hs{homogeneous space}

\newcommand{\x}{{\mathbf{x}}}
\newcommand{\0}{{\mathbf{0}}}
\newcommand{\vv}{{\bf{v}}}

\newcommand{\vy}{{\bf y}}

\newcommand{\p}{{\bf p}}

\newcommand{\vq}{{\bf q}}

\newcommand {\comm}[1]   {\textcolor{red}{#1}}
\newcommand{\re}[1] {\textcolor{violet}{#1}}

\DeclareMathOperator{\codim}{codim}

\newcommand{\vre}{\varepsilon}

\newcommand\nz{\smallsetminus \{0\}}

\newtheorem{thm}{Theorem}[section]
\newtheorem{lem}[thm]{Lemma}
\newtheorem{prop}[thm]{Proposition}
\newtheorem{cor}[thm]{Corollary}

\newtheorem{remark}[thm]{Remark}

\newtheorem{claim}[thm]{Claim}

\begin{document}
\title[Dimension drop conjecture for diagonal flows]
{{On the dimension drop conjecture for \\ diagonal flows on the space of lattices}}
\author{Dmitry Kleinbock}
\address{Department of Mathematics, Brandeis University, Waltham MA}
\email{kleinboc@brandeis.edu}

\author{Shahriar Mirzadeh}
\address{ Department of Mathematics, Brandeis University, Waltham MA}

 \email{shahmir@brandeis.edu}

\begin{abstract} 
Let $X = \ggm$, where $G$ is a Lie group and $\Gamma$ is a lattice in $G$,  let $U$ be an open subset of $X$, {and let $\{g_t\}$ be a one-parameter subgroup of $G$.  
Consider the set of points in $X$ whose $g_t$-orbit misses $U$; it has measure zero if the flow is ergodic. It has been conjectured that this set has Hausdorff dimension strictly smaller than the dimension of $X$. This conjecture is proved when $X$ is compact or when $G$ is a simple Lie group of real rank $1$. In this paper we  prove this conjecture for the case $G=\SL_{m+n}(\R)$, $\Gamma=\SL_{m+n}(\Z)$  and $g_t=\diag (e^{nt}, \dots, e^{nt},e^{-mt}, \dots, e^{-mt})$, in fact providing an effective estimate for the codimension.  The proof uses exponential mixing of the flow together with the  method of integral inequalities for height functions on $\SL_{m+n}(\R)/\SL_{m+n}(\Z)$. We also discuss an application to the problem of improving Dirichlet's theorem in simultaneous \da.}
\end{abstract}

\thanks{The first-named author was supported by NSF grant  DMS-1900560.}

\subjclass[2010]{Primary: 37A17, 37A25; Secondary: 11J13.}
\date{September 2021}

\maketitle
\section{Introduction}
{{Let  $G$ be a 
{Lie group, 
and let $\Gamma $   
be a lattice} in $G$.} 
Denote by $X$ the homogeneous space  $G/\Gamma$ and by $\mu$   the
$G$-invariant probability measure on $X$. 
%
%
For {an unbounded subset $F$ of $G$
and a non-empty open subset $U$ of $X$  define the sets ${E(F,U)}$  and $\widetilde E(F,U)$ as follows:
\eq{set} {\begin{aligned}E(F,U) &:=  \{ x \in X: gx \notin
 U \ \forall\, g\in F\}\\\subset\  \widetilde E(F,U)&:= \{ x \in X: \exists\text{ compact }Q\subset G\text{ such that }gx \notin
 U \ \forall\, g\in F\ssm Q\}\\ &= \bigcup_{\text{compact }Q\subset G}E(F\ssm Q,U)
 \end{aligned}
 }}
of points in $X$ whose 
{$F$}-trajectory always (resp., eventually) stays away from $U$.}
If {$F$ is a subgroup or a subsemigroup of $G$ acting ergodically} 
on $(X,{\mu})$,
then {the trajectory $Fx$ of $x$ 
is dense for {$\mu$-}almost all $x \in X$, in particular $\mu\big(\widetilde E({F},U)\big) = 0
$  whenever $U$ has non-empty interior}.


The present paper studies the following natural question, asked several years ago by Mirzakhani (private communication): 
{if  $E(F,U)$ has measure zero, does it necessarily have less than} full \hd?
In fact it is reasonable to conjecture that the answer is always {`yes'}; in other words, that the following `Dimension Drop Conjecture' holds: {\it if 
$F\subset G$ is a subsemigroup and $U$ is an open subset of $X$, then either $E(F,U)$ has positive measure, or its dimension is less than the dimension of $X$.} {The same can be stated about $\widetilde E(F,U)$.

{If $X$ is compact, or, more generally, if the complement of $U$ is compact, then the dimension drop  {conjecture} follows from the uniqueness of the measure of maximal entropy, see e.g.\ \cite[Theorem 9.7]{MT} and \cite[Proposition 7.5]{KW}. In that case an explicit estimate for the codimension of $E(F,U)$  was recently obtained in \cite{KM}. When $X$ is not compact, the situation is more complicated due to a possibility of the `escape of mass'.
The conjecture is known in the following cases:}
\begin{itemize}
\item  {$F$ consists of   \emph{quasiunipotemt} elements}, that is, {for each $g\in F$} all eigenvalues of $\Ad\, g$ have absolute value $1$. This follows from Ratner's Measure Classification Theorem and the work of Dani and Margulis, see  \cite[Lemma 21.2]{St} and  \cite[Proposition 2.1]{DM}.
\item {$G$ is a simple Lie group of} real rank $1$ \cite{EKP}.
\end{itemize}
{Another example is contained in a recent paper by Guan and Shi \cite{GS}: {extending a method developed earlier in \cite{KKLM},} they proved that for an arbitrary one-parameter subgroup action on a finite-volume homogeneous space the set of points with divergent trajectories (that is,  trajectories eventually leaving any compact subset of the space) has Hausdorff dimension  strictly less than full.} {See also \cite{AGMS, RW} for a related work.}

\smallskip
In this paper we establish {a special case of the aforementioned conjecture for a specific, and important for applications, non-compact \hs\ of a higher rank Lie group, and for a special choice of diagonalizable elements of $G$.
More specifically,  
we fix $m,n\in\N$, let   \eq{sln}{G=\SL_{m+n}(\R), \ \Gamma= \SL_{m+n}(\Z),\ X = \ggm,}}}
and set
\eq{gt}{{F^+ := \{g_t: t\ge 0\}, \text{ where }g_t:=\diag (e^{nt}, \dots, e^{nt},e^{-mt}, \dots, e^{-mt})} . }
We will also 
choose {$a>0$ and consider 
a subsemigroup $F_a^+$ of $F^+$ generated by $g_a$,
that is, let}
\eq{fa}{{F_a^+  :=\big\{ \diag (e^{ant}, \dots, e^{ant},e^{-amt}, \dots, e^{-amt}): t\in\Z_+\big\}.}}

\ignore{{A natural} question one can  ask is: how large can this set of  
measure zero be?  If the semigroup $F^+$ is \emph{quasiunipotemt}, that is, all eigenvalues of $\Ad\, g_1$ have absolute value $1$, then, whenever the action is ergodic and $U$ is non-empty, the set \equ{set} is contained in a countable union of proper submanifolds of $X$ -- this follows from Ratner's Measure Classification Theorem and the work of Dani and Margulis,   {see \cite[Lemma 21.2]{St} and  \cite[Proposition 2.1]{DM}}. On the other hand, if $F^+$ is not quasiunipotent and $U = \{z\}$ for some $z\in X$, it is shown in \cite{K1} that the set $E(F^+,\{z\})$ has full \hd.}

{An important role in the proof will be played by  the \emph{unstable horospherical subgroup} with respect to $F^+$, {namely}}
\eq{h}{H:=\left\{h_s: s \in M_{m,n}
\right\}, \,\,\, \text{where}\,\,\, h_s :=\begin{bmatrix}
I_m & s \\
0 & I_n
\end{bmatrix}.}
Here and hereafter $M_{m,n}$  stands for the space of $m\times n$ matrices with real entries. {It will be repeatedly used in the proof that the conjugation map $h_s\mapsto g_th_sg_{-t}$ corresponds to a dilation of $s$ by $e^{(m+n)t}$.} 

{For the rest of this paper we let $G$, $\Gamma$, $X = \ggm$, 
${F^+_a}$  and $H$ be as in {\equ{sln}--{\equ{h}}}.}
We are going to denote by $\|\cdot\|$ the Euclidean norm on $M_{m,n}$, and will  choose a {right-invariant} Riemannian {structure  on $G$ which  agrees with the one induced by $\|\cdot\|$ on $M_{m,n}\cong \Lie(H)$. 
{If $P$ is a subgroup of $G$, we will denote by $B^{P}(r)$  the open ball of radius $r$ centered at the identity element with respect to the  
metric  on $P$ coming from the Riemannian structure induced from $G$. Also, to simplify notation, ${{B(r)}}$ will stand for the Euclidean ball in $M_{m,n}$ centered at $0$ with radius $r$, so that    $$B^H(r) = \{h_s: s\in M_{m,n},\ \|s\| < r\}= \{h_s: s\in B(r)\}.$$}
We will denote by
`$\dist$' the corresponding Riemannian metric on $G$} 
  and will use the same notation for  the induced metric on $X$.
  
  We need to introduce the following notation: for {an open} subset $U$ of $X$ and $r > 0$ denote by $\sigma_rU$ the \emph{inner $r$-core\/} of $U$, defined as
$$
\sigma_rU :=  \{x\in X: {\dist}(x,U^c) > r\}.$$
{This is an open subset of $U$, {whose measure is close to $\mu(U)$ for small enough values of $r$. 
The latter implies that the quantity
\eq{su1}{\theta_{U}:= \sup \left\{ 0<\theta\le 1: \mu(\sigma_{2 \sqrt{mn}\theta}U) \ge \frac{1}{2} \mu(U) \right \} 
} is positive if $U\ne\varnothing$.} {Also, for a closed subset $S$ of $X$ denote  by  $\partial_rS$ the \emph{$r$-neighborhood\/}   of $S$, that is, 
 $$
\partial_rS :=  \{x\in X: {\dist}(x,S)  < r\}.
$$
Note that we always have $\partial_rS \subset \big(\sigma_r(S^c)\big)^c$. In particular, for $z\in X$ we have $\partial_r\{z\} = B(z,r)$,   the open ball {in $X$} of radius $r$ centered at $z$.} 


\smallskip
{We denote by $\dim E$ the \hd\ of the set $E$, and by $\codim E$  its Hausdorff codimension, i.e.\ the difference between the dimension of the ambient set and the \hd\ of $E$.} 
The next {theorem, which is the main result of the paper, establishes the Dimension Drop Conjecture for the case \equ{sln}--{\equ{fa}}, and, moreover,   does it in a quantitative way,  giving an explicit estimate for the codimension of ${\widetilde E({F^+_a},U)}$ as a function of $U$ and $a$. {In what follows,} the notation ${A\gg B}$,
where  $A$ and $B$ are quantities depending on certain parameters, will mean ${A \ge  CB}$,  
with {$C$ being a constant} 
dependent  only on $m$ and $n$. }

\ignore{\begin{thm}\label{dimension drop 2} Let $G=\SL_{m+n}(\R)$, $\Gamma= \SL_{m+n}(\Z)$, fix $T > 0$  and let $g = g{T}$ where $g_t$ is as in \eq{gt}. Then 
here exist positive constants $K_3, p' \ge 1$, and $r_3 \le 1 $ such that for any open subset $U$ of $\ggm$ and any $r$ satisfying $0<r<\min( \mu (\sigma_r U)^{p'}, r_3)$, the set ${\widetilde E(g,U)}$ has Hausdorff codimension at least
$$ K_3 \frac{\mu(\sigma_r (U ))}{\log \frac{1}{r}}. $$ 
\end{thm} }

\ignore{\begin{thm} \label{delta average}
There exist positive constants $K_3, p' \ge 1$, and $r_3 \le 1 $ independent of $U$ such that for any $0 < \delta \le 1$, and for any $r$ satisfying $0<r<\min( \mu (\sigma_r U)^{p'}, {{r_{2}}})$, the set$$\{ x \in X:x \,\,\, \delta \textsl{-escapes} \,\,U \,\, \textsl{on average}  \}$$ has Hausdorff codimension at least
$$ K_3 \delta \frac{\mu(\sigma_r (U ))}{\log \frac{1}{r}}. $$ 
\end{thm}}

{\begin{thm}\label{dimension drop 2} 
{There exist positive constants ${c},r_1$ 
such that 
 {for any $a>0$ 
 and} for any open subset $U$ of $X$ one has
\eq{dbound}{ \codim {\widetilde E({F^+_a},U)} \gg \frac{\mu(U)}{\log \frac{1}{{r(U,a)}}},}  
where 
{{\eq{cu}{{r(U,a)}:=\min \left( \mu(U),\theta_U,{c}{e^{-a}},r_1 \right).  }  } } 
{In particular, if $U$ is non-empty we always have 
$ \dim {\widetilde E({F^+_a},U)} < \dim X      $.}}
\end{thm} 

{Similarly to previous  papers on the subject, Theorem \ref{dimension drop 2} 
is deduced by considering
the intersection of ${\widetilde E({{F^+_a}},U)}$ with  the orbits $Hx$ of the group $H$.
{{\begin{thm}\label{dimension drop 3} 
{There exist positive constants ${c},r_1$ 
such that 
{for any $a>0$,
any $x \in X$, and} for any open subset $U$ of $X$ one has
{$$ \codim \big(\{h\in H: hx\in {\widetilde E({F^+_a},U)}\} \big) \gg \frac{\mu(U)}{\log \frac{1}{{r(U,a)}}}, $$}  
where $r(U,a)$ is as in \equ{cu}.}
  \end{thm} }
}

\ignore{It what follows, if $P$ is a subgroup of $G$, we will denote by $B^{P}(r)$  the open ball of radius $r$ centered at the identity element with respect to the  
metric  on $P$ coming from the Riemannian structure induced from $G$. In particular, we have $$B^H(r) = \{h_s: s\in M_{m,n},\ \|s\|\le r\}.$$ The ball $B^G(r)$ in $G$ will be simply denoted by 
$B(r)$,
and $B(x,r)$ will denote  the open ball of radius $r$ centered at $x \in X$. Given  $x\in X$, let us denote by $\pi_x$ the map  $G\to X$ given by {$\pi_x(h) = hx$}, and by  $r_0(x)$ the \textsl{injectivity radius} of $x$, defined as $$
\sup\{r > 0: 
\pi_x\text{ is injective on }B(r)\}.$$
If $K\subset X$ is bounded, let us denote by $r_0(K)$ the \textsl{injectivity radius} of $K$: $$
r_0(K) := \inf_{x\in K}r_0(x) = \sup\{r > 0: 
\pi_x\text{ is injective on }B(r)\  \ \forall\,x\in K\}.$$}


{As a special case of the two theorems above,
 in the next corollary the Hausdorff dimension of the set of points whose $g_a$-trajectory misses a 
 small enough neighborhood of a smooth submanifold of $X$ 
 is estimated. 
{\begin{cor}
\label{Cor2}
If $S\subset X$ is a $k$-dimensional embedded smooth submanifold, then there exist  {$\vre_S,c_S,C_S > 0$}
such that for any ${a > 0}$
and any positive  {$\vre < \min(\vre_S, c_S{e^{-a}})$}
  one has
\eq{forS} 
{{\codim  \big(\{h\in H: hx\in {\widetilde E({F^+_a},\partial_{\vre}S)}\}\big)  {\,\ge}\ } C_S \frac{{\vre^{\dim X - k }}}{{\log (1/\vre)}}.}
In addition, if $k=0$ and $S = \{z\}$,  {the  constants $c_S$ and $C_S$} can be chosen \linebreak independent of $z$; that is
 there exist  {$r_z, c_{*} >0$}
 such that {for any ${a > 0}$, 
any $z \in X$  
and any} 
{{$0<\vre<{\min\big(r_z,c_{*} {e^{-a}}\big)}
$}} 
one has 
\eq{forz} {
\codim  \big(\{h\in H: hx\in {\widetilde E\big({F^+_a},B(z,\vre)\}\big)}  {\,\gg} \ \frac{{\mu\big (B(z,\vre)\big)}}{{\log (1/\vre)}}\,.}
Similar estimates hold for the codimension of ${\widetilde E({F^+_a},\partial_{r}S)}$ and  ${\widetilde E\big({F^+_a},B(z,r)\big)}$ in $X$.
\end{cor}}
\begin{remark} \rm It is clear from  \equ{cu} that Theorems  \ref{dimension drop 2} and \ref{dimension drop 3}, as well as Corollary \ref{Cor2}, produce  analogous results for the action of the one-parameter semigroup $F^+$: namely,  by letting $a$ tend to zero one sees that the codimensions of ${\widetilde E({F^+},U)}$ in $X$ and $\{h\in H: hx\in {\widetilde E({F^+},U)}\}$ in $H$ are bounded from  {below} by 
$\frac{\mu(U)}{-\log {{\min \left( \mu(U),\theta_U,r_1 \right)}}}$ {times a  {constant dependent  only on $m,n$}.}
\end{remark}}

{Finally let us describe an application of Theorem  \ref{dimension drop 3} to simultaneous \da. 
{Given $c\ge 1$, say that $s\in \mr$ is \emph{$c$-Dirichlet improvable} 
if 
for all sufficiently large $N$ 
\eq{di}{\begin{aligned}\text{ there
exists $\p\in\Z^m$ and $\vq\in\Z^n\nz$  such that }\\\|s\vq - \p\| < cN^{-n/m} 
\text{ and }0  <  \|\vq\|  \le N.\qquad \end{aligned}}
Here $\|\cdot\|$ stands for the supremum norm on $\R^m$ and $\R^n$.
We let 
$\bold{DI}_{m,n}(c)$ be the set of  $c$-Dirichlet improvable $s\in\mr$.
Note that Dirichlet's theorem (see e.g.\ \cite{S2}) implies that $\bold{DI}_{m,n}(1) = \mr$.
Davenport  and
Schmidt proved \cite{DS} that the Lebesgue measure of $\bold{DI}_{m,n}(c)$  is zero for any $c < 1$. On the other hand, they also showed that $\bigcup_{c<1}\bold{DI}_{m,n}(c)$  contains the set 
of badly approximable $m\times n$ matrices, which is known \cite{S1} to have full \hd.}}

{In recent years much attention has been directed to the set $$\bold{Sing}_{m,n} := \bigcap_{c<1}\bold{DI}_{m,n}(c)$$ of \emph{singular} matrices.  In \cite{KKLM} its \hd\ was estimated from above by $mn \left(1 - \frac{1}{m+n}\right)$, and then in \cite{DFSU1} this estimate was shown to be sharp for any $m,n$ with $\max(m,n) > 1$, verifying a conjecture made in \cite{KKLM}. The case $m=1$ was settled previously  in \cite{CC}. Moreover, it is shown there that for any integer $n\ge 2$ and any $\varepsilon > 0$ for small enough $c$ 
it holds that $$\frac{n^2}{n+1} + c^{n+\varepsilon} \le \dim\big(\bold{DI}_{1,n}(c)\big)\le \frac{n^2}{n+1} + c^{n/2 - \varepsilon}
$$
(see \cite[Theorem 1.3 and Corollary 6.10]{CC} for a more precise estimate).}

{As a corollary from our main result, we deduce that for any $c < 1$ the codimension of $\bold{DI}_{m,n}(c)$
is positive:
\begin{thm} \label{di estimate}  $\dim\big(\bold{DI}_{m,n}(c)\big) < mn$ for any   $c < 1$.
\end{thm}
In fact for $c$ close enough to $1$ with some extra work one can explicitly estimate from below} the codimension of $\bold{DI}_{m,n}(c)$ in $\mr$ as a function of $c$. 
  \smallskip
  
  The structure of the paper is as follows. {Roughly speaking, the proof has  two main ingredients. One deals with orbits staying inside a fixed compact subset of $X$, which are handled in \S\ref{covcpt} with the help of the exponential mixing of the $g_t$-action on $X$ as in \cite{KM}. The other one (\S\S \ref{heightnew}--\ref{height}) takes care of orbits venturing far away into the cusp of $X$; there we use the method of integral inequalities for height functions on $X$ pioneered in \cite{EMM} and thoroughly explored in \cite{KKLM}. The two ingredients are combined in  \S\ref{endofproof} in the form of a covering result (Proposition \ref{first1}). Then in \S\ref{intermediate} the results of the preceding sections are used to derive two separate dimension bounds (Theorem \ref{first}), which are then used in  \S\ref {easyproofs} to prove Theorem \ref{dimension drop 3}. After that we show how the latter implies Theorem \ref{dimension drop 2}, and  use Theorems \ref{dimension drop 2} and  \ref{dimension drop 3} to deduce Corollary \ref{Cor2} and Theorem \ref{di estimate}.}
  

{We remark that the methods of this paper are applicable in much wider generality: in particular, with some modification of the argument the Dimension Drop Conjecture can be established for the action of $$g_t:=\diag (e^{r_1t}, \dots, e^{r_mt},e^{-s_1t}, \dots, e^{-s_nt})$$
on ${\SL_{m+n}(\R)/\SL_{m+n}(\Z)}$ (here $r_1,\dots,r_m$ and $s_1,\dots,s_n$ are positive numbers with 
$\sum_{i=1}^mr_i = \sum_{j=1}^ns_j$),  as well as for diagonalizable flows on \hs s of other semisimple Lie groups. This is going to be addressed in a forthcoming work. {In 
the last section of the paper we list some other generalizations and open questions.}

\section{A covering result for orbits staying in compact subsets of $X$}\label{covcpt}
\ignore{Similarly to the previous  papers on the subject, the
main theorem is deduced from
 {a result} that estimates
\eq{hor es}{\dim E({g},\sigma_{r}U) \cap Hx,} where $x\in X$ and $H$ is the {\sl unstable horospherical subgroup} with respect to $F^+$, {defined 
as
$$ H:=\left\{
\begin{bmatrix}
1 & A \\
0 & 1
\end{bmatrix}: A \in M_{m,n} \right\} .$$}}
For $N \in \N$, for any subset $S $ of $X$, {any $x\in X$}  and any ${t > 0}$ let us define the following set:
 \eq{setsAT}{
 {{A}_x(t,{r},{N},S)
 = {\big\{s\in {{B(r)}} 
 : 
 g_{{it}} h_sx \in S \,\,\,\forall\,i \in \{1,\dots,N  \} \big\}}}.
 }}
For 
{our dimension estimates} it will be useful
to have 
{{a bound on} the number of cubes of sufficiently small {side-length} needed to cover the sets of the above form.
In this section we will consider the case of $S$ being compact, which was 
thoroughly studied in \cite{KM}. We are going to apply \cite[Theorem 4.1]{KM}, which was proved   in the generality of $X = \ggm$ being an arbitrary \hs, and $H$ being a subgroup of $G$ with the Effective Equidistribution Property (EEP) with respect to $F^+$. The latter property was shown there to hold in the case 
\equ{sln}--\equ{gt}, or, more generally, as long as $H$ is the expanding horospherical subgroup relative to $F^+$, and the {$F^+$}-action on $X$ is exponentially mixing. See also \cite{KM1,KM4} for some earlier motivating work on the subject.}

{Here we need to introduce the notion of the injectivity radius of points and subsets of $X$. Given  $x\in X$, let us denote by 
$r_0(x)$ the \textsl{injectivity radius} of $x$, defined as $$
\sup\left\{r > 0: \text{the map }G\to X,\ g\mapsto gx
\text{ is injective on }B^{ G}(r)\right\}.$$
If $K\subset X$ is bounded, we will denote by $$r_0(K):= \inf_{x\in K}r_0(x)$$ the \textsl{injectivity radius} of $K$.}

\smallskip

The following theorem   is an immediate corollary of  \cite[Theorem 4.1]{KM} {applied {to $P = H$, $L = \dim P = mn$ and $U = S^c$}.}
\begin{thm}\label{main cov}There exist   constants
{ $$\ {{0<r_{2}< \frac{1}{16\sqrt{mn}}}},\  b_0\ge 2, \ b\ge 1,  \ {0 < K_1 \le 4}, \ {K_0 
 \ge 1},\,K_2,\,\lambda > 0$$
such that 
for any 
{compact subset $S$} of $X$,
any $0<r<\min \big(r_0(\partial_{1/2}  {S^c}),{r_2} \big)$,}
any $x\in \partial_{
r}{{S}}$, 
 {any $N \in \N$, and any $t\in\R$ satisfying
{\eq{t estimate}{{t}> b_0 +b \log \frac{1}{r},}}}
the set ${{A}_x\left(t,{\frac{r}{16 \sqrt{mn}}},{N},{{S}}\right)}$ can be covered with at most
$${K_0{{e}^{ mn(m+n)Nt}}  \left(1 - K_1 \mu \big({\sigma_{r}}{S^c}\big) +\frac{K_2 e^{- \lambda t}}{r^{mn}}  \right)^N}$$
balls in $
{M_{m,n}}$ of diameter $re^{-(m+n)N t}$.
\end{thm}

{
We are going to apply the above theorem to cover sets of type \equ{setsAT} with cubes of diameter substantially bigger than $re^{-(m+n)N t}$. Namely we will work with cubes of side length $\theta{e}^{-(m+n)N t}$, where {$\theta\in\left[4r ,\frac {1}{2\sqrt{mn}}\right]$.}}

\begin{thm}\label{main cor}
{Let} {$r_{2}$, $b_0$, $b$, $K_0$,  $K_1$, $K_2$ and
$\lambda$} 
{be as in Theorem \ref{main cov}. Then} 
for any compact subset $S$ of $X$, 
any $r>0$ {such that
\eq{r1}{r< \min \big(r_0(\partial_{1} S),{{{r_{2}}}} \big),}}
any $t$ satisfying {\equ{t estimate}, 
}any {{$\theta\in\left[4r ,\frac {1}{2\sqrt{mn}}\right]$}}, any $x\in \partial_{
r}{S}$ and  any $N \in \N$, the set ${{A}_x\left(t,{\frac{r}{32 \sqrt{mn}}},{N},{{S}}\right)}$ can be covered with at most
$${ \left(\frac{4r}{\theta} \right)^{mn} K_0{e}^{ mn(m+n)Nt} \left(1 - K_1 \mu \big(\sigma_{2 \sqrt{mn}{\theta}} S^c\big) +\frac{K_2 e^{- \lambda t}}{r^{mn}}  \right)^N }$$
cubes in $
{M_{m,n}}$ of  {side length} {${\theta}{e}^{-(m+n)N t}$.} 
\end{thm}

\begin{proof}

Let $S$ be a compact subset in $X$, 
let $r$, $t$ and  $N$ be such that conditions {\equ{t estimate} and \equ{r1}} 
are satisfied, and let {{$\theta\in\left[4r ,\frac {1}{2\sqrt{mn}}\right]$}}.
{Let $\mathcal{C}_N$ be} a covering of $
{B\big(\frac{r}{32 \sqrt{mn}}\big)}$ with cubes of {side-length} ${\theta}{e}^{-(m+n)Nt}$ in $
{M_{m,n}}$ whose interiors are disjoint and {whose sides are parallel to the coordinate axes. }
Next, {consider a covering $\mathcal{C}'_N$ of $\cup_{R\in\mathcal{C}_N}R$  
with interior-disjoint cubes of {side-length} $r{e}^{-(m+n)Nt}$ in $
{M_{m,n}}$, also with sides parallel to the coordinate axes}. 
Here and hereafter we will denote by ${\Leb}$ {the 
{Lebesgue} measure on $
{M_{m,n}}$. 
} 

\smallskip
Let {$x \in X$}. We need the following lemma.
\begin{lem} \label{cube}
For any cube  $R$ 
in $\mathcal{C}_N$ which has non-empty intersection with the set ${{A}_x\left(t,{\frac{r}{32 \sqrt{mn}}},{N},{{S}}\right)}$  there exist at least {$ (\frac{\theta}{2r})^{mn}$} cubes in $\mathcal{C}'_N$ which lie in the interior of $R$. Moreover, all such cubes are subset of 
${{A}_x\left(t,{\frac{r}{16 \sqrt{mn}}},{N},{{\partial_{{\sqrt{mn} \theta}} S}}\right)}$.
\end{lem}
\begin{proof}
Observe that any cube in $\mathcal{C}'_N$ that 
{contains a point of $\sigma_{{r{e}^{-(m+n)Nt}}}R$ must lie in  the interior of $R$}.
Therefore, the number of cubes in $\mathcal{C}'_N$ that lie in the interior of $R$ is at least 
$${
\frac{{\Leb} \left( \sigma_{{r{e}^{-(m+n)Nt}}}R \right)}{r^{mn}e^{-mn(m+n)Nt}} 
= \frac{(\theta-2r)^{mn}e^{-mn(m+n)Nt}}{r^{mn}e^{-mn(m+n)Nt}} \ge \left(\frac{\theta}{2r}\right)^{mn}. 
}$$
Now let $B $ be 
{one of those cubes}.  {The side-length} of $R$ is 
{
$$\begin{aligned}{\theta}{e}^{-(m+n)Nt}&\underset{\equ{t estimate}}{<} \theta e^{-b_0(m+n)N}\cdot r^{b(m+n)N}{ \underset{(b_0\ge 2)}\le} \frac 1{2\sqrt{mn}} e^{-2(m+n)}\cdot r\\ & \le{\frac{r}{{32 {mn}}}}(4\sqrt{m}e^{-2m})(4\sqrt{n}e^{-2n})\le{\frac{r}{{32 {mn}}}},\end{aligned}$$
}
 hence its diameter is at most ${\frac{r}{{32 \sqrt{mn}}}}$. Since $R$ has non-empty intersection with $
{B\big(\frac{r}{32 \sqrt{mn}}\big)}$, we have $B \subset R \subset 
{B\big(\frac{r}{16 \sqrt{mn}}\big)}$. Moreover, since by our assumption \linebreak $R \cap {{A}_x\left(t,{\frac{r}{32 \sqrt{mn}}},{N},{{S}}\right)} \neq \varnothing$, we can find ${s} \in R$ such that $g_{i t}{h_s}x \in S$ for all $i \in \{1,\dots,N\}$.  To prove that
 {$B \subset {{A}_x\left(t,{\frac{r}{16 \sqrt{mn}}},{N},{{\partial_{{\sqrt{mn} \theta}} S}}\right)}$}, we need to take any
 ${s}' \in B$ and any $i \in \{1,\dots,N\}$ and show that \eq{goal}{g_{i t}{h_{s'}}x \in {{\partial_{\sqrt{mn}{\theta}}S}}.} 
Clearly \eq{conj}{g_{i t}{h_{s'}}x=(g_{it}{h_{s'-s}} g_{- it})g_{i t}{h_{s}}x , }
{and, since both ${s}$ and ${s'}$ are in $R$, it follows that $${\|s'-s\|\le 
\sqrt{mn}{e}^{-(m+n)Nt}{\theta}},$$ hence} 
 {$g_{it}
 {h_{s'-s}}g_{- i t} \in B^H(\sqrt{mn}{\theta})   \subset  B^{G}(\sqrt{mn}{\theta})$.} Thus, since $g_{i t}{h_s}x \in S$, from \equ{conj} we obtain \equ{goal}, which finishes the proof of the lemma.
\end{proof}
{Now note that every ball of diameter $r {e}^{-(m+n)Nt}$ in $
{M_{m,n}}$ can be covered with at most $2^{mn}$ cubes of side-length $r {e}^{-(m+n)Nt} $ in $\mathcal{C}_N$. Hence, by Lemma \ref{cube} and by Theorem \ref{main cov} applied to $S$ replaced with  {$\partial_{\sqrt{mn}{\theta}} S \subset \partial_{1/2}S$,} {for any $x\in \partial_rS\subset \partial_r({\partial_{\sqrt{mn}{\theta}}S})$} the set ${{A}_x\left(t,{\frac{r}{32 \sqrt{mn}}},{N},{{S}}\right)}$ can be covered with at most 
{\begin{align*}
& 
\left( \frac{2r}{\theta}\right)^{mn} 2^{mn} \cdot K_0 {{e}^{ mn(m+n)Nt}} \left(1 - K_1 \mu \big({\sigma_{r} (\sigma_{{\sqrt{mn}{\theta}}}} S^c ) \big) +\frac{K_2 {e}^{- \lambda  t}}{r^{mn}}  \right)^N  
 \\
& \le \left( \frac{4r}{\theta}\right)^{mn} K_0 {{e}^{ mn(m+n)Nt}} \left(1 - K_1 \mu \big( \sigma_{{2 \sqrt{mn}{\theta}}} S^c  \big) +\frac{K_2 {e}^{- \lambda  t}}{r^{mn}}  \right)^N 
\end{align*}}
cubes in $
{M_{m,n}}$ of side-length $\theta {e}^{-(m+n)Nt}$.}
This finishes the proof.
\end{proof}

{\section{Height functions and non-escape of mass} \label{heightnew}
In the next two  sections we describe trajectories which venture outside of large compact subsets of $X$. The method we are using, based on integral inequalities for height functions, was introduced in a breakthrough paper of Eskin, Margulis and Mozes \cite{EMM}, and later adapted in \cite{KKLM}. Our argument basically follows the scheme developed in the latter paper, with minor modifications.}

{Let $x\in X$ be a lattice in $\R^{m+n}$. Following {\cite{EMM},  say that
a subspace $L$ of $\R^{m+n}$ is $x$-rational if $L \cap x$ is a lattice in $L$, and for any $x$-rational 
subspace $L$, denote by $d_x(L)$  the volume of $L/(L \cap x)$. {Equivalently, let us denote by $\|\cdot\|$ the extension of the Euclidean norm on $\R^{m+n}$ to $\bigwedge(\R^{m+n})$;  then 
\eq{dxl}{d_x(L)=\|v_1  \wedge \cdots \wedge v_i\|,\text{ where  $\{v_1, \dots,v_i\}$  is a $\Z$-basis for }L\cap x.}} 
For  any $i=1,\dots,m+n$ and any $x \in X$ we let $F_i(x)$ denote the set of $i$-dimensional $x$-rational subspaces of $\R^{m+n}$.}

Now for $1\le i \le m+n$ define 
$${\alpha_i(x):= \sup \left\{ \frac{1}{{d_x(L)}}: L \in F_i(x)\right\}.}$$
Clearly $\alpha_{m+n}(x)\equiv1$, and for convenience we also set $\alpha_0(x)\equiv1$ for all $x \in X$. Functions $\alpha_1,\dots,\alpha_{m+n-1}$ can be thought of as height functions on $X$, in the sense that a sequence of points $x_j$ diverges in $X$ (leaves every compact subset) if and only if $\lim_{j\to\infty}\alpha_i(x_j) = \infty$ for some (equivalently, for all) $i = 1,\dots, m+n-1$. This is a consequence of Mahler's Compactness Criterion and Minkowski's Lemma.}
 
{
As in \cite{KKLM}, we will approximate the Lebesgue measure on a neighborhood of identity in $H$ by the Gaussian distribution on {$\mr$}. Namely, we 
will let {${\rho}_ {\sigma ^2}$ denote the 
Gaussian probability measure  on {$\mr$} where each component is i.i.d.\ with mean $0$ and variance $\sigma^2$}. 
 
\smallskip
{In the following theorem, 
which is a simplified version of \cite[Corollary 3.6]{KKLM}, we push forward  the probability measure ${\rho_1}$ from {$\mr$} to the orbit $Hx$, where $x\in X$, and then translate it by $g_t$. {Let us use the following notation: for $x\in X$, $t> 0$ and a measurable function $f$ on $X$ define}
$$
{I_{x,t}(f) := \int_{\mr} {f(g_th_sx)}\,d {\rho_1}(s).}
$$
\begin{thm}\label{main in} {There exists $c_0 \ge 1 $ depending only on $m,n$ with the following property:} for any $t \ge 1$, any $x \in X$, and for any $i \in \{1,\dots,m + n-1 \}$ one
has
\eq{gaussianestimate}{
{I_{x,t}\left(\alpha_i^{1/2}\right)
\le c_0\left(  {e}^{- {t}/{2}} {\alpha_i(x)}^{1/2}    + e^{mnt} \max_{0<j \le \min (m+n-i,i )} 
\sqrt{\alpha_{i+j}(x)^{1/2} \alpha_{i-j}(x)^{1/2}}
\right).}}
\end{thm}

{To make the paper self-contained, we include all the details of the proof.
The first step,  an analogue of \cite[Proposition 3.1]{KKLM}, is to obtain an estimate similar to  \equ{gaussianestimate}, but replace  the height functions $\alpha_i$  with  $\frac{1}{d_x(L)}$, where {$L\in F_i(x)$ is fixed}, and instead of the Gaussian measure ${\rho_1}$ use the probability measure $dk$ on the maximal compact subgroup $K = \text{SO}(m+n)$ of $G$.} Note that in  the argument below all the implicit constants depend only on $m,n$.

\begin{prop}\label{integ es}
For any $t\ge 1$, any $i \in \{1,\dots,m+n-1\}$, and any {decomposable}  ${v= v_1\wedge\cdots\wedge v_i} \in \bigwedge^i (\R^{m+n})$ we have:
$${
\int_K { \| g_tkv  \|^ {-1/2}}\, dk \ll {e}^{-{t}/{2}} {\|v\|^{-1/2}}}.$$
\end{prop}

\begin{proof}
{Notice that $K$ acts transitively on the
set of  decomposable $v\in \bigwedge^i (\R^{m+n})$ with a fixed norm.
Therefore $
{\int_K { \| g_tkv  \|^ {-1/2}}\, dk}$ is a function of $\|v\|$, and from its homogeneity it follows that
$$
{\int_K { \| g_tkv  \|^ {-1/2}}\, dk}=C(t) \|  v\|^{-1/2}    $$
for some function $C:\R_+\to \R_+$. 
Now choose $x_1,\dots,x_i$ to be independent
standard Gaussian $\R^{m+n}$-valued random variables. 
Then we have
$$ \mathbb{E}\left(
{\int_K { \| g_tk(x_1  \wedge \cdots \wedge x_i)  \|^ {-1/2}}\, dk}\right)= C(t)\mathbb{E}( \|x_1 \wedge \cdots \wedge x_i \|^{-1/2}),$$
where the right hand side is finite in view of \cite[Lemma 3.2]{KKLM}. On the other  {hand}, using the $K$-invariance of $x_1,\dots,x_i$ we get
$$ \mathbb{E}\left(
{\int_K { \| g_tk(x_1  \wedge \cdots \wedge x_i)  \|^ {-1/2}}\, dk}\right)= \mathbb{E} \|g_t(x_1 \wedge \cdots \wedge x_i) \|^{-1/2}.$$
Thus to prove the proposition, it suffices to show that
$$ \mathbb{E}\big( \|g_t(x_1 \wedge \cdots \wedge x_i) \|^{-1/2}\big)  \ll  {e}^{- {t}/{2}}.$$}
Let ${V^+} \subset \R^{m+n}$ denote the $m$-dimensional subspace spanned by $e_1, \dots,  e_m$
and let ${V^-}$ be the complementary subspace, so that
$$\|g_tv\|={e}^{nt}\|v\|   ,\  \|g_tw\|={e}^{-mt}\| w\|$$
for $v \in {V^+}$ and $w \in {V^-}$. In particular, for any $v \in \bigwedge^i ({V^+})$ we have $\|g_tv \|={e}^{int}\|v\|$. Let $\pi_u^{(i)} :
\bigwedge ^ i (\R^{m+n}) \rightarrow
\bigwedge^i ({V^+})$
 be the natural (orthogonal)
projection. Clearly, we have:
$$\pi_u^{(i)}
(x_1 \wedge  \cdots  \wedge x_i) = \pi_u^{(1)}(x_1) \wedge
\cdots \wedge \pi_ u ^{(1)}(x_i),$$
where each of $\pi_u^{(1)}(x_j)$ is a standard Gaussian random variable in $m$ dimensions.

We first assume that $i \le m$. Then we have:
{$$ \|g_t(x_1 \wedge \cdots \wedge x_i)   \|   \ge \|\pi_u^{(i)}   g_t(x_1 \wedge \cdots \wedge x_i)\| =  \| g_t\pi_u^{(i)}  (x_1 \wedge \cdots \wedge x_i)\|= {e}^{int} \| \pi_u^{(i)}  (x_1 \wedge \cdots \wedge x_i)\|,$$
hence
\begin{align*}
\mathbb{E}\big(\|g_t(x_1 \wedge \cdots \wedge x_i)\|^{-1/2}\big) 
\le 
{e}^{\frac{-int}{2} } \mathbb{E}\big(\|\pi_u^{(i)}(x_1 \wedge \cdots \wedge x_i)\|^{-1/2}\big) \ll  {e}^{- \frac{t}{2}},
\end{align*}
where in the last inequality we are again using \cite[Lemma 3.2]{KKLM}, i.e.\ the finiteness of $\mathbb{E}( \|x_1 \wedge \cdots \wedge x_i \|^{-1/2})$.} This finishes the proof for $i \le m$. {The case $m < i \le n$ can be handled by duality, following the lines of the proof of \cite[Proposition 3.1]{KKLM}.}\end{proof}

\ignore{Now suppose $m < i \le n$. {We shall handle this case using duality.  Define the linear map:
$$*:{\textstyle\bigwedge}^{i}( \R^{m+n}) \rightarrow {\textstyle\bigwedge}^{m+n-i}( \R^{m+n})    $$
by $*(\wedge_{j\in I}e_j)=\wedge_{{j} \notin I}e_j$. 
It is a $K$-invariant isometry of ${\textstyle\bigwedge}( \R^{m+n})$ 
with the property that 
$
 *(g_tv)=g_{-t}(*v)
$ for any $t\in\R$ and $v\in {\textstyle\bigwedge}( \R^{m+n})$.
Therefore,
$$\int_K \|g_tkv \|^{-\frac{1}{2}}\,dk = \int_K \|*g_tkv\|^{-\frac{1}{2}}\, dk = \int_K \|g_{-t}k(*v)\|^{-\frac{1}{2}}\,dk  .   $$
This concludes the proof by applying the above for $i'=m+n-i$, $n'=m,$ and  $m'=n.$}}

{Let us introduce the following notation: if $h\in {G}$, we will denote by $\|h\|_\infty$ the norm of $h$ viewed as an operator on  $\bigwedge (\R^{m+n})$. We note that $\|h\|_\infty=\|h^{-1}\|_\infty$ for any $h\in H$, since {$h = h_s$ and $h^{-1} = h_{-s}$} are conjugate by $\begin{pmatrix}I_m & 0\\0&  -I_n\end{pmatrix}$. That is, 
\eq{norminverse}{{\|h_s\|^{-1}_\infty \|v\| \le \|h_sv\| \le\|h_s\|_\infty \|v\| \quad\text{for any  }s\in\mr}\text{  and }v \in \textstyle\bigwedge (\R^{m+n}).}
Note that {$\|h_s\|_\infty$ grows polynomially in $s$: more precisely, \eq{polgr}{\|h_s\|_\infty\ll \|s\|^{\min(m,n)}.}}}
 {We will also use a norm estimate similar to \equ{norminverse} but for the $g_t$-action: 
 \eq{omega}{e^{-mnt}\|v\|\le \|g_tv\|\le e^{mnt}\|v\| \quad\text{for any  }t \ge 1\text{  and }v \in \textstyle\bigwedge (\R^{m+n}).}
The next lemma, which is a special case `$\beta = 1/2$' of \cite[Lemma 3.5]{KKLM}, shows that Proposition \ref{integ es} will remain valid if integration over $K$ is replaced with integration over a bounded subset of {$\mr$}.
\begin{lem}\label{lem:KUasymp} There exists a neighborhood $W$ of {$0$ in $\mr$ such that for any $s_0 \in   \mr$}, $t\ge 1$,  $i \in \{1,\dots,m+n-1\}$, and decomposable $v \in \textstyle\bigwedge^i (\R^{m+n})$ we have
$${\int_{s_0 + W} \| g_t h_s v \|^{-1/2}\,ds \ll\| h_{s_0} \|_\infty^{1/2} \int_K \| g_t k v \|^{-1/2} \,dk.}$$
\end{lem}
\begin {proof}[Proof of Theorem \ref{main in}]Fix  $x \in X$ and  $i \in \{1,\dots,m + n-1 \}$.
Let $L_0\in F_i(x)$ be such that
\begin{equation}\label{eqn:Li}
\alpha_i(x)=\frac{1}{{d_x(L_0)}}.
\end{equation}
Note that in view of \equ{norminverse} and
 \equ{omega} we have
\eq{alphai}{\begin{aligned}\alpha_i(g_thx)&\le\frac{1}{{d_x(g_thL_0)}}\le e^{mnt}\frac{1}{{d_x(hL_0)}}\\ & \le  {e^{mnt}} \| h \|_\infty\frac{1}{{d_x(L_0)}} \le e^{mnt} \| h\|_\infty \alpha_i(x).\end{aligned}}

We shall consider two cases.
\begin{itemize}
\item[\bf Case 1.] The {subspace $L_0$ is an outlier, that is, ${d_x(L_0)}$ is much smaller   than ${d_x(L)}$ for any $L \in F_i(x)$ different from $L_0$. Namely,}
$$
{d_x(L)} \ge e^{2mnt} {d_x(L_0)}\quad \forall\,L \in F_i(x) \smallsetminus\{L_0\}.
$$
%
Then for any $L \in F_i(x) \smallsetminus\{L_0\}$ and $h\in H$  in view of \equ{norminverse} and
 \equ{omega} we have 
$${d_x(hL_0)} \le \|h\|_\infty {d_x(L_0)} \le e^{-2mnt}\|h\|_\infty  {d_x(L)} \le  e^{-2mnt}\|h\|_\infty^2{d_x(hL)},
$$
hence
\begin{equation*}\label{eqn:uniqueLi}
{d_x(g_thL_0)}\le e^{mnt}{d_x(hL_0)}  \le  e^{-mnt}\|h\|_\infty^2 {d_x(hL)}\le\|h\|_\infty^2 {d_x(g_thL)}.
\end{equation*}
Therefore $\displaystyle \alpha_i(g_t h x) \le \frac{\|h\|_\infty^2}{{d_x(g_thL_0)}}$ and
\begin{equation}\label{eqn:1pls}
{
{I_{x,t}\left(\alpha_i^{1/2}\right)}\le  \int_{{\mr}}\|h_s\|_\infty {d_x(g_th_sL_0)}^{-1/2} \,d{\rho_1}(s)}.
\end{equation}
Take {$W\subset \mr$ as in Lemma \ref{lem:KUasymp}. Clearly, for any $s' \in \mr$
\eq{multline}{\begin{aligned} 
&\int_{s'+W}\|h_s\|_\infty{d_x(g_th_sL_0)}^{-{1/2}} \,d {\rho_1}(s) \\
\ll\ &\left( \max_{s \in s'+W}\|h_s\|_\infty e^{-\frac{\|s\|^2}{2}}\right)\int_{h_0W}d_x(g_th_sL_0)^{-{1/2}}\,ds\\
\underset{ \equ{polgr}}\le  &\ e^{-\frac{\|s'\|^2}{2}+O\left(\|s'\|\right)}\int_{h_0W}d_x(g_th_sL)^{-{1/2}}\,ds,
\end{aligned}}
where the implied constant is independent of $s'$. Summing over a lattice $\Lambda$ in $\mr$  sufficiently fine so that $\mr=W+\Lambda$,  
 we conclude that 
\begin{align*}
\int_{\mr}\|h_s\|_\infty d_x(g_th_sL_0)^{-1/2} \,d{\rho_1}(s) &\le \sum_{s' \in \Lambda} \int_{s'+W}\|h_s\|_\infty d_x(g_th_sL_0)^{-1/2} \,d{\rho_1}(s) \\
&\underset{\equ{multline}}\ll \sum_{{s'} \in \Lambda} e^{-\frac{\|s'\|^2}{2}+O\left(\|s'\|\right)} \int_{s'+W} d_x(g_th_sL_0)^{-1/2} \,d{\rho_1}(s)\\
\text{(by Lemma~\ref{lem:KUasymp})}& \ll \sum_{{s'} \in \Lambda} \| h_{s'} \|_\infty^{1/2}e^{-\frac{\|s'\|^2}{2}+O\left(\|s'\|\right)} \int_K d_x(g_t k L_0)^{-1/2} \,dk\\
& \ll \int_{K} d_x(g_t k L_0)^{-{1/2}}\,dk.
\end{align*}  
Thus,  \eqref{eqn:1pls} and Proposition~\ref{integ es}  give
$${I_{x,t}\left(\alpha_i^{1/2}\right)}\ll e^{-t/2}d_x(  L_0)^{-1/2}\underset{\eqref{eqn:Li}}= e^{-t/2} \alpha_i(x)^{1/2}.$$}
\item[\bf Case 2.] There exists $L \in F_i(x)$ different from  $ L_0$ such that \eq{smaller}{ {d_x(L)} < e^{2mnt} {d_x(L_0)}.}
Let $j$ 
be the 
{dimension of $L/(L\cap L_0) \cong (L+L_0)/L_0$; then the dimension of  $L+L_0$ is equal to $i+j$. Note that we have \eq{emm}{d_x(L) d_x(L_0) \ge d_x(L\cap L_0) d_x(L+L_0),} see \cite[Lemma~5.6]{EMM}.   
Then for any  $h\in H$ we can write
\begin{multline*}
\alpha_i(g_t h x) \underset{\equ{alphai}}\le  e^{mnt} \| h \|_\infty \alpha_i(x) \underset{\eqref{eqn:Li}}=\frac{e^{mnt} \| h \|_\infty}{d_x(L_0)}\underset{\equ{smaller}}<\frac{e^{2mnt} \| h \|_\infty}{\sqrt{d_x(L) d_x(L_0)}} \\ \underset{\equ{emm}}\le \frac{e^{2mnt} \| h \|_\infty}{\sqrt{d_x(L\cap L_0) d_x(L+L_0)}} \le e^{2mnt} \| h \|_\infty \sqrt{\alpha_{i+j}(x) \alpha_{i-j}(x)}.
\end{multline*}
Hence
$$
{I_{x,t}\left(\alpha_i^{1/2}\right)}
\le e^{mnt} \max_{0<j\le \max(m+n-i,i)}\big( 
{\alpha_{i+j}(x) \alpha_{i-j}(x)} \big)^{1/4} \int_{\mr}\| h_s \|_\infty^{1/2}\, d{\rho_1}(s).
$$}
\end{itemize}
It follows from \equ{polgr} that $$ {\int_{\mr}\| h_s \|_\infty^{1/2}\,d{\rho_1}(s)} \ll 1,$$ hence combining the above two cases establishes \equ{gaussianestimate} {with some uniform $c_0$}.
\end{proof}}

{An immediate application of Theorem \ref{main in} is obtained via the `convexity trick' introduced in \cite{EMM} and formalized in \cite{KKLM}:  from \equ{gaussianestimate} and 
\cite[Proposition 4.1]{KKLM} with $\beta_i = 1/2$ for each $i$ it follows that for any $t\ge 1$ there exist positive constants
{$\omega_0 = \omega_0(t),\dots,\omega_{m+n}= \omega_{m+n}(t)$} and $C_0 $ such that the linear combination
\eq{tildold}{
\tilde{\alpha}   :=\sum_{i=0}^{m+n} \omega_i {\alpha_i}^{1/2}}
satisfies
$$
I_{x,t}(\tilde\alpha) \le 2c_0   {e}^{- {t}/{2}} \tilde{\alpha}(x)    + C_0
$$
for all $x\in X$. However, for our purposes it will be necessary to get precise expressions for the constants $\omega_0,\dots,\omega_{m+n}$ and $C_0$.
This forces us to go through the argument 
from \cite{EMM} and \cite{KKLM} adapted for this special case.} 
Namely, take \eq{defeps}{\vre = \vre(t) =\frac{e^{-(mn+1/2)t}}{m+n-1},} for $i \in \{0, \dots, m+n \}$  define $p(i):=i(m+n-i)$, and let 
$${\omega_i(t) := \vre^{p(i)} =  \frac{e^{-(mn+\frac12)i(m+n-i)t}}{(m+n-1)^{i(m+n-i)} }.}$$
This gives rise to the height function of the form \equ{tildold} which we are going to use in the later sections. Since it depends on the (fixed) parameter $t$, with some abuse of notation 
we will denote it by 
\eq{tild}{\tilde{\alpha}^{{t}}   :=\sum_{i=0}^{m+n} \omega_i (t){\alpha_i}^{1/2} = \sum_{i=0}^{m+n}  \frac{e^{-(mn+1/2)i(m+n-i)t}}{(m+n-1)^{i(m+n-i)} }{\alpha_i}^{1/2}.}
A key role in our proof will be played by subsets $X$ consisting of points $x$ with large (resp., not so large) values of $\tilde{\alpha}^{{t}}(x)$. Namely, for $M > 0$ let us define
{\eq{xtm}{X_{> M}^t:= \{x\in X : \tilde{\alpha}^{{t}}(x)> M\}\text{ and }X_{\le M}^t:= \{x\in X : \tilde{\alpha}^{{t}}(x)\le M\}.}
Since $ \tilde{\alpha}^{{t}}$ is proper, the sets $X_{\le M}^t$ are compact, and  $X_{> M}^t$ are `cusp neighborhoods'  with compact  complements.}

\smallskip
Observe that for any $i,j$ such that $0<j \le \min \{i, m+n-i  \}$  we have 
$$2p(i)-p(i+j)-p(i-j)=2i(m+n-i) - (i+j)(m+n-i-j) - (i-j)(m+n-i+j) = 2j^2.$$
Then for each $i\in \{1,\dots,m + n-1 \}$ the inequality \equ{gaussianestimate} implies
$$
\begin{aligned}
&I_{x,t}\left(\omega_i\alpha_i^{1/2}\right) 
\le c_0\vre^{p(i)}\left( {e}^{- {t}/{2}} {\alpha_i(x)}^{1/2}    +  e^{mnt} \max_{0<j \le \min (m+n-i,i )} \sqrt{\alpha_{i+j}(x)^{1/2} \alpha_{i-j}(x)^{1/2}}\right)\\
= & \ c_0\vre^{p(i)} {e}^{- {t}/{2}} {\alpha_i(x)}^{1/2}    + c_0\vre^{j^2}e^{mnt} \max_{0<j \le \min (m+n-i,i )} \sqrt{\vre^{p(i+j)}\alpha_{i+j}(x)^{1/2} \vre^{p(i-j)}\alpha_{i-j}(x)^{1/2}}\\
\le  &\ c_0\omega_i {e}^{- {t}/{2}} {\alpha_i(x)}^{1/2}    + c_0\vre e^{mnt} \max_{0<j \le \min (m+n-i,i )} \sqrt{\omega_{i+j}\alpha_{i+j}(x)^{1/2} \omega_{i-j}\alpha_{i-j}(x)^{1/2}} .
\end{aligned}
$$
Since both $\omega_{i+j}\alpha_{i+j}(x)^{1/2}$ and $\omega_{i-j}\alpha_{i-j}(x)^{1/2}$ are not greater than $\tilde{\alpha}^{{t}}(x)$, we obtain 
{\eq{ixtprelim}{
\begin{aligned}
I_{x,t} (\tilde{\alpha}^{{t}}) = \   I_{x,t}&\left(2 +\sum_{i=1}^{m+n-1} \omega_i {\alpha_i}^{1/2} \right) 
\le 2 + \sum_{i=1}^{m+n-1} I_{x,t}\left( \omega_i {\alpha_i}^{1/2} \right)\\
=  \ \ \ 2 \ &+ \ c_0  {e}^{- {t}/{2}}\tilde{\alpha}^{{t}}(x)     +(m+n-1) c_0\vre(t) e^{mnt} \tilde{\alpha}^{{t}}(x).
\end{aligned}
}}
Thereby we have arrived at 
\begin{prop}\label{Hieght in}
Let $\tilde{\alpha}^{{t}}$ be defined by 
\equ{tild},
and let $c_0$ be as in Theorem \ref{main in}.
 Then: 
 \begin{itemize} \item[\rm (a)] For any $t \ge 1$ any $x \in X$ one
has
\eq{unit in}{
\begin{aligned}
I_{x,t} (\tilde{\alpha}^{{t}}) 
\le  2 + 2c_0  {e}^{- {t}/{2}}\tilde{\alpha}^{{t}}(x). 
\end{aligned}
}
 \item[\rm (b)] For any $t \ge 1$ and 
 any {$x \in X_{> {{e}^{t/2}}/{c_0}}^t$} 
 we have:
\eq{unit in 2} {I_{x,t} (\tilde{\alpha}^{{t}}) \le 4c_0 {e}^{- {t}/{2}} \tilde{\alpha}^{{t}}(x).} 
 \end{itemize} 
\end{prop}
 \begin{proof} \equ{unit in} is obtained {from \equ{ixtprelim} via the substitution \equ{defeps}}.
Part (b) is immediate from (a) since $\tilde{\alpha}^{{t}}(x) \ge  \frac{{e}^{t/2}}{c_0}
$ is equivalent to $2 \le 2c_0  {e}^{- {t}/{2}}\tilde{\alpha}^{{t}}(x)$.
 \end{proof}}}

{\begin{remark}\label{calpha} \rm
Note that it follows from \equ{dxl} and the definition of functions $\alpha_i$  that for any $i = 0,\dots,m+n$, $h\in G$ and $x\in X$ one has 
$$\frac1{ \|h\|_\infty}\alpha_i(x)\le \alpha_i(hx)\le \|h^{-1}\|_\infty \alpha_i(x).$$
Since $\tilde{\alpha}^{{t}}$ is a linear combination of functions $\alpha_i^{1/2}$, it  satisfies similar inequalities. Specifically, in what follows we are going to take $h$ from the ball $B(2)$ of radius $2$ in $G$. Let us define 
$$C_\alpha:= \sup_{h\in B(2)} \max\big(\|h\|_\infty, \|h^{-1}\|_\infty\big)^{1/2};$$
then it is clear that for any $h  \in B(2)$  and any $x \in X$ we have:
\eq{C alpha}{C_\alpha^{-1} \tilde{\alpha}^{{t}}(x) \le \tilde{\alpha}^{{t}}({h}x) \le C_\alpha \tilde{\alpha}^{{t}}(x).}
\end{remark}}

\section{Covering results for the orbits visiting non-compact part of $X$} \label{height}
{
 In the following proposition, which is the main result of this section, we will {fix $x \in X$, $k,N\in \N$ and $t,M > 0$, and will} work with the set
\eq{mainset}{\begin{aligned}A_x\big(kt,1,N,g_tX_{> C_\alpha 
M}^t\big) = \left\{  s \in {{B(1)}}:g_{i{k}t}h_sx \in g_tX_{> C_\alpha 
M}^t \,\,\,\, \forall\, i \in \{1,\dots, N  \} \right\}\\  =\left\{  s \in {{B(1)}}:\tilde{\alpha}^{{t}}(g_{(ik-1)t} h_sx) > {C_\alpha}
M \,\,\,\, \forall\, i \in \{1,\dots, N  \} \right\},\qquad\qquad\end{aligned}}} 
where $C_\alpha$ is as in Remark \ref{calpha}.

\begin{prop}\label{main pro}
There exists $C_{1} \ge 1$ such that {for any $2 \le k \in \N$, any ${t \ge 2}$,}
any $N \in \N$, any $x \in X$, and for any $M \ge  C_\alpha  e^{\frac{mnt}{2}}$
we have
$$ \int_{
{A_x\left(kt,1,N,g_tX_{> C_\alpha 
M}^t\right)}
}          \tilde{\alpha}^{{t}}(g_{Nkt}h_sx) \,d s \le{\left( (k-1) C_{1}^{k} {e}^{- \frac{t}{2}}\right)^N} \max\big(\tilde{\alpha}^{{t}}(x),1\big).$$
. 

\end{prop}
\begin{proof}
{
{{Let us fix $x, k, t, N$ and $M$} as in the statement of the proposition; the sets defined in the course of the proof will depend on these parameters.}   
 Define 
 $${Z_{M}:=\left\{(s_1,\dots,s_k) \in  {{{B(1)}}}^k: \tilde{\alpha}^{{t}} (g_{t}h_{s_{k-1}} \cdots g_t h_{s_1}x) > M \right\}.}$$
Then we can write
{\eq{first-in}{\begin{aligned}
&\int \cdots \int_{Z_{C_\alpha^{-1}M}} \tilde{\alpha}^{{t}}(g_th_{s_k} \cdots g_th_{s_{1}}x) \,d {\rho}_1(s_k) \cdots d {\rho}_1(s_1)   \\
& {=} \ignore{\le} \int \cdots \int_{(M_{m,n})^{k-1}} 1_{
{X_{> C_\alpha^{-1}M
}^t}
}(g_{t}h_{s_{k-1}} \cdots g_t h_{s_1}x)\cdot I_{g_{t}h_{s_{k-1}} \cdots g_t h_{s_1}x,t} (\tilde{\alpha}^{{t}})  \,d {\rho}_1(s_{k-1}) \cdots d {\rho}_1(s_1) \\
& \underset{\equ{unit in 2} }\le 4c_0 {e}^{-\frac{t}{2}} \int \cdots \int_{(M_{m,n})^{k-1}} 
 \tilde{\alpha}^{{t}}({g_{t}h_{s_{k-1}} \cdots g_t h_{s_1}x})   \,d {\rho}_1(s_{k-1}) \cdots d {\rho}_1(s_1),
\end{aligned}}
{where {$c_0$ is as in {Theorem} \ref{main in}}. Note that the use of Proposition \ref{Hieght in} in the last step is justified since $C_\alpha^{-1}M  \ge e^\frac{mnt}{2}  \ge {{e}^{t/2}}/{c_0}$.}
Next, by using \equ{unit in} $(k-1)$ times we get:
\eq{second-in}{\begin{aligned}
&
 \int \cdots \int_{(M_{m,n})^{k-1}} 
 \tilde{\alpha}^{{t}}({g_{t}h_{s_{k-1}} \cdots g_t h_{s_1}x}) 
\,d {\rho}_1(s_{k-1}) \cdots d {\rho}_1(s_1) \\
&
 \le (2c_0 {e}^{-\frac{t}{2}})^{k-2}  \tilde{\alpha}^{{t}}(x)+2\big((2c_0 {e}^{-\frac{t}{2}})^{k-2}+\cdots +1\big)
\\
&
 \le (2c_0 )^{k-2}  \tilde{\alpha}^{{t}}(x)+2(k-2)(2c_0 )^{k-2}\le 4(k-1)(2c_0 )^{k-2} \max\big( \tilde{\alpha}^{{t}}(x),1\big).
\end{aligned}
}
So by combining \equ{first-in} and \equ{second-in} we have:
$$
\int \cdots \int_{Z_{C_\alpha^{-1}M}} \tilde{\alpha}^{{t}}(g_th_{s_k} \cdots g_th_{s_{1}}x) \,d {\rho}_1(s_k) \cdots d {\rho}_1(s_1)  
 \le 8(k-1)(2c_0)^{k-1} {e}^{- \frac{t}{2} } \max\big(\tilde{\alpha}^{{t}}(x),1\big).
$$
}

{Now define the function $\phi: {B(1)}^{k}\to M_{m,n}$ by
\eq{defphi}{{\phi(s_1,\dots, s_k):=\sum_{j=1}^k {e}^{-(m+n)(j-1)t}s_j  }.}
Note that \eq{phi}{g_th_{s_k} \cdots g_th_{s_{1}}= g_{kt}h_{\phi(s_1,\dots,s_k)}}
We will need the following observation:
\begin{lem}\label{lem as phi}
For any 
$M> 0$,
 $\phi^{-1}\big(\phi (Z_{M} )\big) \subset Z_{C_\alpha^{-1} M}$.
\end{lem}
\begin{proof}
Let $(s_1,\dots,s_k) \in  {B(1)}^{k}$ be such that $\phi(s_1,\dots,s_k) \in \phi(Z_{M})$. Then there exists $(s'_1,\dots,s'_k) \in  Z_{M}$ such that
$\phi(s_1,\dots,s_k) =\phi(s'_1,\dots,s'_k)$.
Hence, using \equ{phi} we get
$$g_th_{s_k} \cdots g_th_{s_{1}}=g_th_{s'_k} \cdots g_th_{s'_{1}},$$
which implies
$$g_t h_{s_{k-1}} \cdots g_t h_{s_1}= h_{s'_k-s_k}  g_t h_{s'_{k-1}} \cdots g_t h_{s'_1} .      $$
Note that $h_{s'_k-s_k} \in B^H(2).$ Therefore, by \equ{C alpha} we have 
$$\tilde{\alpha}^{{t}}(g_t h_{s_{k-1}} \cdots g_t h_{s_1}x)\ge C_\alpha^{-1} \tilde{\alpha}^{{t}}(g_th_{s'_{k-1}} \cdots g_t h_{s'_1}x) >  C_\alpha^{-1} M.$$ Hence, $(s_1,\dots,s_k) \in Z_{C_\alpha^{-1} M}$, which finishes the proof of the lemma.
\end{proof}
Using the above lemma we obtain
{
\eq{ineq1}{
\begin{aligned}
&  \int \cdots\int_{ {B(1)}^{k}} 1_{\phi(Z_{M})}\big(\phi(s_1,\dots,s_k) \big)\tilde{\alpha}^{{t}}(g_{kt}h_{\phi(s_1,\dots,s_k)x}) \,d {\rho}_1(s_k) \cdots d {\rho}_1(s_1) \\
& \le \int \cdots\int_{Z_{C_\alpha^{-1}M}} \tilde{\alpha}^{{t}}(g_th_{s_k} \cdots g_th_{s_{1}}x) \,d {\rho}_1(s_k) \cdots d {\rho}_1(s_1) \\
& \le 8(k-1)(2c_0)^{k-1}{e}^{- \frac{t}{2} }\max(\tilde{\alpha}^{{t}}(x),1) .
\end{aligned}}}
To convert the above multiple integral to a single integral, we will use the following}
{\begin{lem} \label{int}
There exists $0 < \Xi <1$ such that for any positive measurable function $f$ on $\mr$ and any \eq{epsdelta1}{0 <  \vre \le \frac{1}{8},\ 0 \le \delta <1} we have 
  $$ { \iint_{{{B(1)}}^2}
   f(\vre x +y) \,d \rho_{1+ \delta ^2}(x) d \rho_1(y) \ge \Xi \cdot \int_{
   {{B(1)}}} f(z) \,d \rho_{1 + {\vre}^2(1+ \delta^2)}(z). }$$
\end{lem}

\begin{proof}
{
Let $\vre$ and  $\delta$ be as in \equ{epsdelta1}. 
{For convenience denote $\sigma := \sqrt{1+ \delta^2}$.} Consider the {change of variables} 
 $$(z,v) 
 :=\left(\vre x +y,\frac{x}{\sigma}-\vre \sigma y\right),$$ or, equivalently \eq{v1}{x= \frac{\sigma (v+\vre \sigma z)}{1+\vre^2{\sigma^2}},\quad y=\frac{z-\vre \sigma v}{1+\vre^2{\sigma^2}}.} 
It is easy to verify that 
 \eq{jac1}{\left|\frac{\partial(z,v)}{\partial(x,y)}\right| = \left(\frac{1+\vre^2{\sigma^2}}\sigma\right)^{mn}}
 and { \eq{squares1}{\frac{
 \|x\|^2}{\sigma^2} +
  \|y\|^2= \frac{
   \|z\|^2+ 
    \|v\|^2}{1+\vre^2{\sigma^2}}.}}
Denote 
{$$\mathcal{D}:= \big \{{(z
,v) \in (M_{m,n})^2}
:
\,\, \|z\|\le 1,\ \|v\|\le 1/4,\ z_{ij}v_{ij}  \ge 0 \,\,\, \forall \,i \in \{1, \dots, m\},\ j\in \{1, \dots, n \}\big \}.$$}
 It readily follows from \equ{v1} that \eq{inD1}{(z,v)\in \mathcal{D}\quad\Longrightarrow\quad 
 {\|x\|\le 1 \text{ and } 
 \|y\|\le 1}.}
Therefore {for any $f$ one has}
$${
 \begin{aligned}
 \iint_{ {B(1)}^{2}}
  f(\vre x +y) \,d \rho_{1+ \delta ^2}(x) d \rho_1(y)
 =\ &\frac1{(2\pi\sigma)^{mn}} \iint_{ {B(1)}^{2}} 
 f(\vre x +y) e^{-\left(\frac{
 \|x\|^2}{2\sigma^2}+\frac{
 \|y\|^2}{2}\right)}\,dx\,dy  \\
  \underset{\equ{jac1},\,\equ{squares1},\,\equ{inD1}}\ge\ & 
   \frac1{\left( 2\pi(1+\vre^2{\sigma^2})\right)^{mn}}
  \iint_{\mathcal{D}} f(z) e^{-\frac{
  \|z\|^2 + \|v\|^2}{2{(1+\vre^2{\sigma^2})}}}\,dz\,dv\\
     \ge & \  \rho_{1+ \vre^2\sigma^2}\left({\left[0,\frac{1}{4 \sqrt{mn}}\right]}^{mn}\right) \cdot
  \int_{ {{B(1)}}} f(z) 
  \,d\rho_{1+ \vre^2\sigma^2}(z)\\
  \underset{\equ{epsdelta1}}\ge &\rho_{
  {33}/{32}}\left({\left[0,\frac{1}{4 \sqrt{mn}}\right]}^{mn}\right) \cdot
  \int_{ {{B(1)}}} f(z) 
  \,d\rho_{1+ \vre^2\sigma^2}(z).
 \end{aligned}
}$$
}
 \end{proof}
}

{Define 
$\sigma_i(t):=\sqrt{\sum_{j=1}^{i-1}e^{-2(m+n)jt}}$ for any $i \in \N$.
Since  $e^{-(m+n)t} \le \frac{1}{8}$ 
{because of the} assumption  
{$t \ge 2$},  for any $i \in\N$ 
we have $\sigma_i(t) < 1$. Hence, by using Lemma \ref{int} 
$(k-1)$ times with $\vre=e^{-(m+n)t}$ and $\delta=\sigma_1(t),\dots, \sigma_{k-1}(t)$ respectively we get
{
$${
\begin{aligned}
&  \Xi^{k-1} \int_{{B(1)}}
1_{\phi(Z_{M})}(s) \tilde{\alpha}^{{t}}(g_{kt}h_sx)\,d {\rho}_{1+\sigma_{k}(t) ^2}(s)                      \\
& =  \Xi^{k-1}  \int_{{B(1)}}
1_{\phi(Z_{M})}(s) \tilde{\alpha}^{{t}}(g_{kt}h_sx)\,d {\rho}_{1+\vre^2(1+\sigma_{k-1}(t) ^2)}(s)         \\
& \le  \int \cdots\int_{{B(1)^k}} 1_{\phi(Z_{M})}\big(\phi(s_1,\dots,s_k) \big)\tilde{\alpha}^{{t}}(g_{kt}h_{\phi(s_1,\dots,s_k)}) \,d {\rho}_1(s_k) \cdots d {\rho}_1(s_1) \\
& \underset{\equ{ineq1}}{\le} 8(k-1)(2c_0)^{k-1}{e}^{- \frac{t}{2} }\max\big(\tilde{\alpha}^{{t}}(x),1\big). \end{aligned}}$$}
Hence, 
\eq{prod}{\int_{{B(1)}}
1_{\phi(Z_{M})}(s) \tilde{\alpha}^{{t}}(g_{kt}h_sx)\,d {\rho}_{1+\sigma_{k} (t)^2}(s)  \le \frac{8(k-1)(2c_0)^{k-1}}{\Xi^{k-1}}{e}^{- \frac{t}{2} }\max\big(\tilde{\alpha}^{{t}}(x),1\big).}
Also, 
since $1+\sigma_{k}(t)^2\in[1,2]$, $d {\rho}_1$ is absolutely continuous with respect to $d {\rho}_{1+\sigma_{k}(t)^2}$ {with a uniform    (over ${{B(1)}}$) bound on the Radon-Nikodym derivative.} 
Thus, we can find ${c_1}\ge 1$ such that \equ{prod} takes the form:
\eq{in 0}  {
\begin{aligned}
& \int_{{B(1)}} 1_{\phi(Z_{M})}(s)\tilde{\alpha}^{{t}}(g_{kt}h_sx)\,d {\rho}_1(s)\le  \frac{8c_1(k-1)(2c_0)^{k-1}}{\Xi^{k-1}}{e}^{- \frac{t}{2} }  \max\big(\tilde{\alpha}^{{t}}(x),1\big). 
\end{aligned}}}

{Now 
{consider the set
$$
A_x\left(tk,1,1,g_tX_{> M}^t\right)=\left\{s \in {{B(1)}}: \tilde{\alpha}^{{t}}(g_{(k-1)t}h_sx) > M \right\}.$$}
It is easy to see that if $s  \in 
{A_x\left(tk,1,1,g_tX_{> M}^t\right)}$, then $$s= \phi(s,0, \dots, 0)\text{ and }(s,0, \dots, 0) \in Z_{M},$$ where $0$ is the zero matrix. 
Hence, \equ{in 0} implies
\eq{ind in}{ \int_{
{A_x\left(tk,1,1,g_tX_{> M}^t\right)}} \tilde{\alpha}^{{t}}(g_{kt}h_sx)\,d {\rho}_1(s)                        \le  \frac{8c_1(k-1)(2c_0)^{k-1}}{\Xi^{k-1}}{e}^{- \frac{t}{2} } \max\big(\tilde{\alpha}^{{t}}(x),1\big). }
\ignore{where $C_\alpha \ge 1$ is such that for any $x \in X$ and any $h \in B^H(2)$ we have:
\eq{C alpha}{C_\alpha^{-1} \alpha(x) \le \alpha(hx) \le C_\alpha \alpha(x).}}
\ignore{Next, given $M>0$, let us define:
$$Z'_{M}:= \{(s_1,\dots,s_N) \in   {(M^1_{m,n})}^N: \tilde{\alpha}^{{t}} ( g_{(k-1)t}h_{s_{i}} \cdots g_{kt} h_{s_1}x) \ge M \,\,\, \forall \,i \in \{1,\dots,N  \}   \} . $$
Since $M \ge  
{e^\frac{mnt}{2}}$, {in view of \equ{omega}} for any 
$y \in X$ one has 
\eq{newimpl}{\tilde{\alpha}^{{t}}(g_{(k-1)t}
y) \ge M  \quad\Longrightarrow\quad\tilde{\alpha}^{{t}}(g_{kt}
y)\ge 1.}
 Thus, by using  \equ{ind in} repeatedly we get for any $N \in \N$ 
\eq{se} {
\begin{aligned}
& \int \cdots \int_{  Z'_{M}} \tilde{\alpha}^{{t}}(g_{kt}h_{s_{N}}\cdots g_{kt}h_{s_1}x) \,d {\rho}_1(s_N) \cdots d {\rho}_1(s_1) \\
& \le   
\left( \frac{8c_1(k-1)(2c_0)^{k-1}}{\Xi^{k-1}} \right)^N{e}^{- \frac{Nt}{2}}
 \max\big(\tilde{\alpha}^{{t}}(x),1\big),   
\end{aligned}}}}
{Next, given $M>0$ and $i \in \N$, let us define:
$$
\begin{aligned}
Z'_{M,i}:= \big \{ & (s_1,\dots,s_i) \in   {(M^1_{m,n})}^i:  \\
&  \tilde{\alpha}^{{t}} ( g_{(k-1)t}h_{s_{j}} g_{kt} h_{s_{j-1}} \cdots g_{kt} h_{s_1}x) > M \,\,\, \forall \,j \in \{1,\dots,i  \}   \big\} .
\end{aligned}$$
Note that
\eq{eq1}{Z'_{M,1}= 
{A_x\left(tk,1,1,g_tX_{> M}^t\right)}.}
Since $M \ge  
{e^{\frac{mnt}{2}}}$, {in view of \equ{omega}} for any 
$y \in X$ one has 
\eq{newimpl}{\tilde{\alpha}^{{t}}(g_{(k-1)t}
y) > M  \quad\Longrightarrow\quad\tilde{\alpha}^{{t}}(g_{kt}
y)> 1.}
Then for any $2 \le i \in \N$, we obtain the following:
\eq{induc}{
\begin{aligned}
& \int \cdots \int_{  Z'_{M,i}} \tilde{\alpha}^{{t}}(g_{kt}h_{s_{i}}\cdots g_{kt}h_{s_1}x) \,d {\rho}_1(s_i) \cdots d {\rho}_1(s_1) \\
& =\int \cdots \int_{  Z'_{M,i-1}} \int_{
{A_{g_{kt}h_{s_{i-1}} \cdots g_{kt} h_{s_1}x}\left(tk,1,1,g_tX_{> M}^t\right)}
} \tilde{\alpha}^{{t}}(g_{kt}h_{s_{i}}\cdots g_{kt}h_{s_1}x) \,d {\rho}_1(s_i) \cdots d {\rho}_1(s_1) \\
&  \underset{\equ{ind in}}{\le} \int \cdots \int_{  Z'_{M,i-1}}  \frac{8c_1(k-1)(2c_0)^{k-1}}{\Xi^{k-1}}{e}^{- \frac{t}{2} } \cdot \max \big(\tilde{\alpha}^{{t}}(g_{kt}h_{s_{i-1}}\cdots g_{kt}h_{s_1}x),1 \big) \,d {\rho}_1(s_{i-1}) \cdots d {\rho}_1(s_1) \\
&  \underset{\equ{newimpl}}{=}\frac{8c_1(k-1)(2c_0)^{k-1}}{\Xi^{k-1}}{e}^{- \frac{t}{2} } \int \cdots \int_{  Z'_{M,i-1}}  \tilde{\alpha}^{{t}}(g_{kt}h_{s_{i-1}}\cdots g_{kt}h_{s_1}x) \,d {\rho}_1(s_{i-1}) \cdots d {\rho}_1(s_1).
\end{aligned}
}
 Thus, by using  \equ{induc} repeatedly we get for any $N \in \N$ 
\eq{se} {
\begin{aligned}
& \int \cdots \int_{  Z'_{M,N}} \tilde{\alpha}^{{t}}(g_{kt}h_{s_{N}}\cdots g_{kt}h_{s_1}x) \,d {\rho}_1(s_N) \cdots d {\rho}_1(s_1) \\
& \le   
\left( \frac{8c_1(k-1)(2c_0)^{k-1}}{\Xi^{k-1}} \right)^{(N-1)}{e}^{- \frac{(N-1)t}{2}}
  \int_{  Z'_{M,1}} \tilde{\alpha}^{{t}}(g_{kt}h_{s_{1}}x) \,d {\rho}_1(s_1)\\
 & \underset{\equ{ind in}, \, \equ{eq1}}{\le }  \left( \frac{8c_1(k-1)(2c_0)^{k-1}}{\Xi^{k-1}} \right)^{N}{e}^{- \frac{Nt}{2}} \max \big(\tilde{\alpha}^{{t}}(x),1\big).
\end{aligned}}}
{Now, similarly to \equ{defphi}, define the function $\psi: {{B(1)}}^{N}\to M_{m,n}$ by
$$ \psi(s_1,\dots, s_N):=\sum_{j=1}^N e^{-(m+n)(j-1)kt}s_j,                   $$ 
so that
\eq{comm eq}{g_{kt}h_{s_N} \cdots g_{kt}h_{s_{1}}= g_{Nkt}h_{\psi(s_1,\dots,s_N)}.}
The following lemma is a modification of Lemma \ref{lem as phi} applicable to the sets $Z'_{M,N}$:
\begin{lem}\label{lem as}
For any 
$M> 0$,  $\psi^{-1}\big(\psi (Z'_{M,N} )\big) \subset Z'_{C_\alpha M,N}$.
\end{lem}
\begin{proof}
Let 
$(s_1,\dots,s_N) \in {{B(1)}}^{N}$ be such that $\psi(s_1,\dots,s_N) \in    \psi( Z'_{C_\alpha M,N})$. Then for some $(s'_1,\dots,s'_N) \in  Z'_{C_\alpha M,N}$ we have:
$$\psi(s_1,\dots,s_N) =\psi(s'_1,\dots,s'_N)$$
Hence, by using \equ{comm eq} we get:
$$g_{kt}h_{s_N} \dots g_{kt}h_{s_{1}}=g_{kt}h_{s'_N} \dots g_{kt}h_{s'_{1}}.$$
Thus, it is easy to see that for any $1 \le i \le N$
\eq{instep}{g_{kt} h_{s_i} \cdots g_{kt} h_{s_1}=h_{\psi_i(-s_{i+1},\dots,-s_{N})+ \psi_i(s'_{i+1},\dots,s'_{N})}  \big( g_{kt} h_{s'_i} \cdots g_{kt} h_{s'_1}\big)       ,}
where for any $(w_{i+1},\dots,w_N) \in {{B(1)}}^{N-i}$ we put
$$\psi_i(w_{i+1},\dots,w_N):=\sum_{j=i+1}^N {e}^{-(m+n)(j-i)kt}w_j .$$
{Note that since $t\ge 2$, one has} $\psi_i(w_{i+1},\dots,w_N) \in {{B(1)}}$ for any $(w_{i+1},\dots,w_N) \in {{B(1)}}^{N-i}$. Hence, in view of \equ{instep}, for any $1 \le i \le N$ we have 
$$g_{kt}h_{s_N} \cdots g_{kt}h_{s_{1}} \in B^H(2)g_{kt} h_{s'_i} \cdots g_{kt} h_{s'_1},$$ which, {since $(s'_1,\dots,s'_N) \in  Z'_{C_\alpha M,N}$}, implies
 $(s_{1},\dots,s_{N}) \in Z'_{ M,N}$. This finishes the proof of the lemma.
\end{proof}}}
{Now by combining \equ{se} and Lemma \ref{lem as} we get:
{
\eq{in1}
{\begin{aligned}
& \int \cdots \int_{ {{B(1)}}^{N} } 1_{ \psi(Z'_{C_\alpha M,N})}(\psi(s_1,\dots,s_N))\tilde{\alpha}^{{t}}(g_{Nkt}h_{\psi(s_1,\dots,s_N)}x) \,d {\rho}_1(s_N) \cdots d {\rho}_1(s_1) \\
& \le    
 \left( \frac{8c_1(k-1)(2c_0)^{k-1}}{\Xi^{k-1}} \right)^{N} {e}^{- \frac{Nt}{2} } \max\big(\tilde{\alpha}^{{t}}(x),1\big). 
\end{aligned}}
{
Then, as before, one can use Lemma \ref{int} 
$(N-1)$ times with $\vre=e^{-(m+n)kt}$ and $\delta=\sigma_1(kt),\dots, \sigma_{N-1}(kt)$ respectively and  obtain:
\eq{int1}
{\begin{aligned}
&  \Xi^{N-1}  \int_{{B(1)}} 1_{\psi( Z'_{C_\alpha M,N})}(s)\tilde{\alpha}^{{t}}(g_{Nkt}h_sx)d {\rho}_{1+\sigma_{N}(kt)^2}(s)    \\
& = \Xi^{N-1}  \int_{{B(1)}}1_{\psi( Z'_{C_\alpha M,N})}(s)\tilde{\alpha}^{{t}}(g_{Nkt}h_sx)d {\rho}_{1+\vre^2(1+\sigma_{N-1}(kt)2)}(s)    \\
& \le  \int \cdots \int_{ {{{{B(1)}}}^N} } 1_{ \psi(Z'_{C_\alpha M,N})}(\psi(s_1,\cdots,s_N))\tilde{\alpha}^{{t}}(g_{Nkt}h_{\psi(s_1,\dots,s_N)}x) \,d {\rho}_1(s_N) \cdots d {\rho}_1(s_1) \\
& \underset{\equ{in1}}{\le}   \left( \frac{8c_1(k-1)(2c_0)^{k-1}}{\Xi^{k-1}} \right)^{N} {e}^{- \frac{Nt}{2} } \max(\tilde{\alpha}^{{t}}(x),1).   
\end{aligned}
}}
} Thus, we get
$${\int_{{B(1)}} 1_{\psi( Z'_{C_\alpha M})}(s)\tilde{\alpha}^{{t}}(g_{Nkt}h_sx)\,d {\rho}_{1+\sigma_{N}(kt)^2}(s) \le   \frac{\left(8c_1(k-1)(2c_0)^{k-1}\right)^N}{\Xi^{kN-1}}  {e}^{- \frac{Nt}{2} }\max\big(\tilde{\alpha}^{{t}}(x),1\big).}$$
Now 
{observe that, in view of \equ{mainset}, if $s \in 
A_x\left(kt,1,N,g_tX_{> C_\alpha M}^t\right)$}, then $$s= \psi(s,0,\dots,0)\text{ and }(s,0,\dots,0) \in Z'_{C_\alpha M,N}.$$ 
Thus, \equ{int1} can be written as
 \eq{last it}{\int_{
 {A_x\left(kt,1,N,g_tX_{> C_\alpha M}^t\right)}} \tilde{\alpha}^{{t}}(g_{Nkt}h_sx)\,d {\rho}_{1+\sigma_{N}(kt)^2} (s)  \le  \frac{\left(8c_1(k-1)(2c_0)^{k-1}\right)^N}{\Xi^{kN-1}}  {e}^{- \frac{Nt}{2} } \max\big(\tilde{\alpha}^{{t}}(x),1\big).}
 Again, 
since $1+\sigma_{N}(kt)^2\in[1,2]$, $d s$ is absolutely continuous with respect to $d {\rho}_{1+\sigma_{N}(kt)^2}$ {with a uniform    (over ${{B(1)}}$) bound on the Radon-Nikodym derivative.} 
Thus, we can find ${c_2}\ge 1$ such that \equ{last it} takes the form
$$ \int_{
{A_x\left(kt,1,N,g_tX_{> C_\alpha M}^t\right)}} \tilde{\alpha}^{{t}}(g_{Nkt}h_sx)\,d s                       \le  \frac{c_2\left(8c_1(k-1)(2c_0)^{k-1}\right)^N}{\Xi^{kN-1}}  {e}^{- \frac{Nt}{2} } \max\big(\tilde{\alpha}^{{t}}(x),1\big).$$
Now define {$C_{1}:= 16 c_0{c_1}{c_2}/\Xi$.} Then by the above inequality we have:
$$ \int_{
{A_x\left(kt,1,N,g_tX_{> C_\alpha M}^t\right)}} \tilde{\alpha}^{{t}}(g_{Nkt}h_sx)\,d s                       \le {\left( (k-1) C_{1}^{k} {e}^{- \frac{t}{2}}\right)^N}\max\big(\tilde{\alpha}^{{t}}(x),1\big).$$
This ends the proof of the proposition. }
\end{proof}
As a corollary we get the following covering result:
\begin{cor}\label{fin cor}
There exists $C_{1} \ge 1$ such that for any  {$\theta \in (0,\frac{1}{\sqrt{mn}}]$,}  {any $2 \le k \in \N$, any {$t\ge 2$},}
any $M \ge C_\alpha^3 e^{mnt} $, any $N \in \N$, and  any $x \in X$, the set 
$${A_x\left(kt,1,N,X_{>   M}^t\right) = \left\{  s \in {{B(1)}}:\tilde{\alpha}^{{t}}(g_{ikt} h_sx) > M \,\,\, \forall\, i \in \{1,\dots, N  \} \right\}}$$ 
can be covered with at most 
$$\frac{C_\alpha {}}{\theta^{mn}} {\left( (k-1)C_{1}^{k} {e}^{(mn(m+n)k-\frac{1}{2})t} \right)^N} \cdot \frac{\max\big(\tilde{\alpha}^{{t}}(x),1\big)}{M} $$
cubes of {side-length}  ${\theta}{e}^{-(m+n)Nkt}$ in {$ M_{m,n}$}.
\end{cor}
\begin{proof}
{Let $x,\theta,M, N, t$ and $k$ be as above, and}
take $C_{1}$ as in Proposition \ref{main pro}. 
{Applying the latter} with $M$ replaced with $C_\alpha^{-2} M {e}^{-\frac{mnt}{2}}$, we have:
 \eq{step}{\int_{
 {A_x\left(kt,1,N,g_tX_{> C_\alpha^{-1} M {e}^{-{mnt}/{2}}}^t\right)}} \tilde{\alpha}^{{t}}(g_{Nkt}h_sx)\,ds \le  {\left( (k-1) C_{1}^{k} {e}^{- \frac{t}{2}}\right)^N}\max(\tilde{\alpha}^{{t}}(x),1).             } 
{In view of \equ{omega}} we have 
{$ X_{>   C_\alpha^{-1}M}^t \subset g_tX_{> C_\alpha^{-1} M {e}^{-{mnt}/{2}}}^t$},
hence 
\eq{step1}{{\begin{aligned} & C_\alpha^{-1} M  \cdot {\Leb} \Big(
A_x\left(kt,1,N,X_{>  C_\alpha^{-1} M}^t\right)
 \Big)      \le   \int_{
 A_x\left(kt,1,N,X_{>  C_\alpha^{-1} M}^t\right)} \tilde{\alpha}^{{t}}(g_{Nkt}h_sx)\,ds \\
& \le \int_{
A_x\left(kt,1,N,g_tX_{> C_\alpha^{-1} M {e}^{-{mnt}/{2}}}^t\right)} \tilde{\alpha}^{{t}}(g_{Nkt}h_sx)ds.
\end{aligned}}}
Thus, using \equ{step} and \equ{step1} we obtain
\eq{nu measure}{ {{\Leb} \Big(
A_x\left(kt,1,N,X_{>  C_\alpha^{-1} M}^t\right)
 \Big)}   \le C_\alpha  {\left( (k-1) C_{1}^{k} {e}^{- \frac{t}{2}}\right)^N} \cdot  \frac{{ \max(\tilde{\alpha}^{{t}}(x),1) }}{ M}      .}
Take a covering of 
{$ {{B(1)}}$} with interior-disjoint cubes of {side-length} ${\theta}{e}^{-(m+n)Nkt}$ in 
{$ M_{m,n}$}. Now let $B$  {be} one of the cubes in this cover which has non-empty intersection with $
{A_x\left(kt,1,N,X_{>   M}^t\right)} $, and let ${s} \in {B \cap 
A_x\left(kt,1,N,X_{>   M}^t\right)} $. Then {$$\tilde{\alpha}^{{t}}(g_{ikt}{h_s}x)> M\text{   for all }1 \le i \le N.$$} {On the other hand},  for any $ {s'} \in B$ and   any   $1 \le i \le N$ one has $${\begin{aligned}g_{ikt }h_{s'}x=\left( g_{ikt } 
h_{s'-s}g_{-ikt} \right) g_{ikt } h_s x &\in B^H(\sqrt{mn}{\theta})g_{ikt }h_sx\\ &\subset B^H(1)g_{ikt } h_sx \subset B(1)g_{ikt}h_sx.\end{aligned}}$$ Hence, we can conclude that
\eq{cube inc}{B \subset  
{A_x\left(kt,1,N,X_{>   C_\alpha^{-1}M}^t\right).}}
Thus, by \equ{nu measure} and \equ{cube inc}, the set $
{A_x\left(kt,1,N,X_{>   M}^t\right)}$ can be covered with at most  
$$  \frac{{\Leb}{\Big(
A_x\left(kt,1,N,X_{>   C_\alpha^{-1}M}^t\right)\Big)}}{\left( \theta {e}^{-(m+n)Nkt}\right)^{mn}} \le \frac{C_\alpha }{\theta^{mn}} {\left( (k-1)C_{1}^{k} {e}^{(mn(m+n)k-\frac{1}{2})t} \right)^N}  \cdot   \frac{{ \max\big(\tilde{\alpha}^{{t}}(x),1\big) }}{ M}.$$
cubes of {side-length}  ${\theta}{e}^{-(m+n)Nkt}$ in 
{$ M_{m,n}$}. This finishes the proof.
\end{proof}

\section{
{The main covering result}}\label{endofproof}
For any $t > 0$, let us define the {compact subset $Q_{t}$ of $X$} as follows:
\eq{qt}{Q_{t}:=X_{\le   C_{\alpha}^3 e^{mnt}}^t.}
In the following lemma we obtain a lower bound for the injectivity radius of the set $\partial_1 Q_t$.
\begin{lem}\label{inj rad} There exist ${0 < C_{2}\le 1}$ and {$p \ge m+n$} independent of $t$ such that for any $t>0$:
$${
r_0(\partial_1 Q_{t}) \ge {C_{2}}{{e}}^{-p t}.
}$$
\end{lem}
\begin{proof} 
Let $t>0$.
Note that {in  view of \equ{C alpha} we have
\eq{p1}{ \partial_1 Q_{t} \subset
X_{\le   C_{\alpha}^4 e^{mnt}}^t;}
then, using \equ{tild} we can write
$$ 
\begin{aligned}
X_{\le   C_{\alpha}^4 e^{mnt}}^t & \subset 
 \left \{x \in X: \alpha_1(x) \le  \frac{e^{-2(mn+\frac12)(m+n-1)t}}{(m+n-1)^{2(m+n-1)} } {C_\alpha ^{8} e^{2mnt} } \right\}  \\
& = \left \{x:\frac{1}{\alpha_1(x)} \ge C_4 {e}^{- \left(2(mn+\frac12)(m+n-1) +2 mn \right)t} \right \},  
\end{aligned} $$
where {$C_4=\frac{1}{C_{\alpha}^{8} (m+n-1)^{2(m+n-1)}}$}. Recall that $\frac{1}{\alpha_1(x)}$ is equal to the norm of the shortest vector in the lattice $x$; therefore by \cite[Lemma 7.2]{KM}, $r_0\left(X_{\le   C_{\alpha}^4 e^{mnt}}^t\right)$ is at least {${C_{2} {e}^{- pt}}$, where $$p={\left((m+n)^2-1 \right) \cdot \big(2(mn+1/2)(m+n-1) +2 mn  \big) \ge m+n}$$} and ${0 < C_{2}\le 1}$ is only dependent on $m $  and $n$. Thus 
we have $r_0(\partial_1 Q_t) \ge {C_{2} e^{-pt}}$, which finishes the proof.}
\end{proof}

\ignore{Define $r_2:=\min(r_0,r_1)$. }

The following {proposition} is 
{our most important covering result}. 
 \begin{prop}\label{first1} 
There exist constants
{$$\ p \ge m+n,  {{0<r_{2}< \frac{1}{16 \sqrt{mn}}}},b_0\ge 2,\ b\ge 1,\ 0 < C_{2} \le 1, \ C_0,C_{3},K_1,K_2,\lambda > 0$$}
such that for any open subset $U$ of $X$ and all integers $N$ and $k \ge 2$ the following holds:
  for all ${t \ge2}$ 
  and all $0<r<1$ satisfying
{\eq{ineq beta3}{{ {e^{ \frac{b_0 - kt}{b}}
} \le r \le \min( C_{2}{e}^{-p  t}, {r_{2}})},}} all {$\theta\in\left[4r ,\frac {1}{2\sqrt{mn}}\right]$}, and for all {$x \in \partial_r \left(Q_t \cap U^c \right)$}, the set ${A_x \left(kt,{\frac{r}{32 \sqrt{mn} }},N,U^c\right)}$ can be covered with at most 
$$\frac{C_0}{\theta^{2mn}} {e}^{mn(m+n)Nkt}   \left(1-  K_1 \mu ({{{\sigma _{2 \sqrt{mn}{\theta}}}{ U}}})+\frac{K_2 {e}^{-\lambda kt}}{r^{mn}} +{\frac{k-1}{\theta^{mn}}}C_{3}^k  {e}^{-\frac{t}{4}} \right)^N     $$
cubes of {side-length} $ \theta {e}^{-(m+n)Nkt}  $
in $M_{m,n}$.
\ignore{\item
There exists a function $C: X \rightarrow \R^+$ such that for all $0<r<1$, all $0<s<1$, all $t \ge t_0 $, and for all $x \in X$, the set ${A}(kt,{\frac{r}{32 \sqrt{mn} }},Q_t^c,{N},x)$ can be covered with at most 
$$\frac{C(x)}{\theta^{mn}}(k-1)^N{C_{1}}^{kN} t^{kN} e^{(mn(m+n)Nk-\frac{N}{2})t}  $
cubes of {side-length}  $\theta e^{-(m+n)Nkt}$ in $ H$.}
\end{prop}
\begin{proof} {The strategy of the proof consists of combining Theorem \ref{main cor} with  Corollary \ref{fin cor}. Recall that the former estimates the number of cubes needed to  cover  the set of points whose trajectories visit a given compact set $S$, while the latter does the same for  trajectories visiting the set $X_{>   M}^t$ which is the complement of a large compact subset of $X$. 
Our goal now is to have a similar result for points whose trajectories visit the set $U^c$, which is not compact and may have a tiny complement. This is done by an inductive procedure which is inspired by the methods introduced in \cite{KKLM}.}

{Take $t \ge 2$} {and let $C_{2}$ and $p$ be as in Lemma \ref{inj rad}}.  Let $0<r<1$ and $2 \le k \in \N$ be such that 
\equ{ineq beta3} is satisfied, where $b_0,b,{{r_{2}}}$ are as in Theorem \ref{main cor}.

Now let {$x \in  \partial_r \left(Q_t \cap U^c \right),$} $N \in \N$, and  {$\theta\in\left[4r ,\frac {1}{2\sqrt{mn}}\right]$}.
 Recall that
$${A_x \left(kt,{\frac{r}{32 \sqrt{mn} }},N,U^c\right)}=\left\{ s \in {{B\left(\frac{r}{32 \sqrt{mn}} \right)}}
: g_{\ell kt}h_sx \in  U^c \,\,\, \forall\,  \ell \in \{ 1,\dots,N \} \right\}.$$
Our goal is to cover ${A_x \left(kt,{\frac{r}{32 \sqrt{mn} }},N,U^c\right)}$ with cubes of {side-length} ${\theta}{{e}}^{-(m+n)Nkt}$ in $M_{m,n}$.
 For any $s \in {A_x \left(kt,{\frac{r}{32 \sqrt{mn} }},N,U^c\right)}$, let us define:
$${J_s}:=\big\{j \in \{1,\dots,N\}:g_{jkt}h_sx \in Q_{t}^c\big\},$$ and for any $J \subset \{ 1,\dots,N  \}$, set:
$$Z(J):=\left\{ s \in {A_x \left(kt,{\frac{r}{32 \sqrt{mn} }},N,U^c\right)}: {J_s}=J \right\}.$$
Note that
\eq{union}{              {A_x \left(kt,{\frac{r}{32 \sqrt{mn} }},N,U^c\right)} =\bigcup_{ J \subset \{   1,\dots,N\} } Z(J)}
Now, set
\eq{C1}{D_1:=  
1-  K_1 \mu ({{{\sigma _{2 \sqrt{mn}{\theta}}}{ U}}})+\frac{K_2{{e}}^{-\lambda kt}}{r^{mn}} 
}
and
\eq{C2}{D_2:=(k-1)C_{1}^k  {{e}^{-t/2}  }            ,}
where $K_1,K_2,\lambda$ are as in Theorem \ref{main cor} and  $C_{1} $ is as in Corollary \ref{fin cor}.
\smallskip

 Let $J$ {be a subset of} $ \{1,\dots, N   \}$. We can decompose  $J$ and $I:=\{1,\dots,N\}  \ssm J$  into sub-intervals of maximal size $J_{1}, \dots,J_{q}$ and $I_{1}, \dots,I_{q'}$ so that
$$J=\bigsqcup_{j=1}^{q} J_{j} \text{ and }I=\bigsqcup_{i=1}^{q'} I_{i}.$$ Hence, we get a partition of the set $\{1,\dots,N\}$ as follows:
$$\{1,\dots,N \}=  \bigsqcup_{j=1}^{q} J_{j} \sqcup   \bigsqcup_{i=1}^{q'} I_{i}                  .$$

Now we inductively prove the following 
\begin{claim}\label{claim}
For any integer $L \le N$, if
\eq{induc eq}{\{1,\dots,L\}= \bigsqcup_{j=1}^{\ell} J_{j}
\sqcup \bigsqcup_{i=1}^{\ell '} I_{i},}
then the set $Z(J)$ can be covered with at most: 
\eq{L case}{
 \left(\frac{C_\alpha^2  }{\theta^{mn}}\right)^{d'_{J,L}+1}\left((2^9 {mn})^{mn} K_0 \right) ^{d_{J,L}+1}{{e}}^{mn(m+n)Lkt}
  D_1^{\sum_{i=1}^{\ell'} |I_i|-d_{J,L}} 
   D_2^ {\sum_{j=1}^{\ell } |J_j|}
} 
cubes of {side-length} ${\theta}{{e}}^{-(m+n)Lkt}$ in $M_{m,n},$ where $K_0$ is as in Theorem \ref{main cor}, and $d_{J,L}$, $d'_{J,L}$ are defined as follows:
$$d_{J,L}:=\# \{i \in \{1,\dots, L\}: \ i<L, \,i \in J  \  and\  i+1 \in I\},$$ 
$$d'_{J,L}:=\# \{i \in \{1,\dots, L\}:\ i<L,\,i \in I \  and\  i+1 \in J \}.$$ 
\end{claim}
{Note that equivalently one can define $$d_{J,L} = \begin{cases}\ell&\text{ if }L\notin J\\
\ell-1&\text{ if }L\in J\end{cases}$$ as the number of intervals in $J\cap\{1,\dots,L\}$ with right endpoints $<L$, and, likewise, 
$$d'_{J,L} = \begin{cases}\ell'&\text{ if }L\notin I\\
\ell'-1&\text{ if }L\in I\end{cases}$$ as the number of intervals in $I\cap\{1,\dots,L\}$ with right endpoints $<L$.}

\begin{proof}[Proof of Claim \ref{claim}] We argue by induction on $\ell + \ell'$. 
When $\ell + \ell' = 1$, we have   $d_{J,L}=d'_{J,L}=0$, and
there are two cases: either $\ell = 1$ and $\{1,\dots,L\}=J_{1}$, or  $\ell' = 1$ and $\{1,\dots,L\}=I_{1}$.
In the first case
$$
\begin{aligned}Z(J) &\subset \left\{ s \in {A_x \left(kt,{\frac{r}{32 \sqrt{mn} }},N,U^c\right)}: g_{ikt}h_sx \in Q_{t}^c \,\,\, \forall\,  i \in \{1,\dots,L\} \right\}\\ &\subset  A_x\left(kt,{\frac{r}{32 \sqrt{mn} }},L,Q_t^c\right) \subset  A_x\left(kt,1,L,X_{>  C_{\alpha}^3 e^{mnt}}^t\right),\end{aligned}$$
where the last step is due to the bound \equ{ineq beta3} on $r$. Therefore, Corollary \ref{fin cor} applied with {$M=C_{\alpha}^3 e^{mnt}$} and $N = L$
shows that this set  can be covered with at most 
$$\begin{aligned}\frac{C_\alpha {}}{\theta^{mn}} \left( (k-1)C_{1}^{k} {e}^{(mn(m+n)k-\frac{1}{2})t} \right)^L    \frac{\tilde{\alpha}^{{t}}(x)}{C_{\alpha}^3 e^{mnt}} 
\underset{\equ{p1}}\le &  \frac{C_\alpha {}}{\theta^{mn}}  \left( (k-1)C_{1}^{k} {e}^{(mn(m+n)k-\frac{1}{2})t} \right)^L     \frac{C_{\alpha}^4 e^{mnt}}{C_{\alpha}^3 e^{mnt}}\\  = \ & \frac{C_\alpha^2 {}}{\theta^{mn}} \left( (k-1)C_{1}^{k} {e}^{(mn(m+n)k-\frac{1}{2})t} \right)^L \end{aligned}$$
cubes of {side-length}  ${\theta}{e}^{-(m+n)Lkt}$ in {$ M_{m,n}$}. Clearly it is bounded from above by \equ{L case} which takes the form 
$$
\frac{C_\alpha^2  }{\theta^{mn}}(2^9 {mn})^{mn} K_0 {{e}}^{mn(m+n)Lkt}    \left((k-1)C_{1}^k  {{e}^{-t/2}  }\right)^L  .
$$
In the second case
$$
\begin{aligned}
Z(J) &\subset \left\{ s \in {A_x \left(kt,{\frac{r}{32 \sqrt{mn} }},N,U^c\right)}: g_{ikt}h_sx \in Q_{t} \,\,\, \forall\,  i \in \{1,\dots,L\} \right\} \\ &\subset  A_x\left(kt,{\frac{r}{32 \sqrt{mn} }},L,{U^c} \cap Q_t\right).\end{aligned}$$
By Lemma \ref{inj rad},  {for any  $U\subset  X$ we have}{$$r_0 \big(\partial_1(U^c \cap Q_t )\big) \ge r_0(\partial_1 Q_t) \ge C_{2} e^{-pt}.$$} So it is easy to see that since condition \equ{ineq beta3} is satisfied, condition \equ{t estimate}  with $t$ replaced by $kt$ and condition \equ{r1}  with {$S$ replaced by $U^c \cap Q_t$} are satisfied as well. Hence we can apply Theorem \ref{main cor}  with {$S$ replaced by $U^c \cap Q_t$}, $N$ replaced by $L$, and $t$ replaced with $kt$. This produces a covering of $A_x \left(kt,{\frac{r}{32 \sqrt{mn} }},N,U^c\right)$ by
$$
\begin{aligned}
&{ \left(\frac{4r}{\theta} \right)^{mn} K_0{e}^{ mn(m+n)Lt} \left(1 - K_1 \mu \big(\sigma_{2 \sqrt{mn}{\theta}} (U \cup Q_t^c)\big) +\frac{K_2 e^{- \lambda t}}{r^{mn}}  \right)^L }\\ \le \ & \frac{C_\alpha^2  }{\theta^{mn}}(2^9 {mn})^{mn} K_0 {{e}}^{mn(m+n)Lkt}\left(1 - K_1 \mu \big(\sigma_{2 \sqrt{mn}{\theta}} (U )\big) +\frac{K_2 e^{- \lambda t}}{r^{mn}}  \right)^L\end{aligned}$$
cubes of {side-length}  ${\theta}{e}^{-(m+n)Lkt} $, finishing the proof of the base of the induction.

\smallskip
\ignore{In the first step of the induction, if $\{1,\dots,L\}=I_{1}$, we have $\sum_{j=1}^{\ell '} |J_j|=0$, $d_{J,L}=0$, and $d'_{J,L}=0$. 
Furthermore, by definition
$$ 
\begin{aligned}
&\left\{ s \in M_{m,n}^{{r}/{32 \sqrt{mn} }}: g_{\ell kt}h_sx \in  U^c \cap {Q_t} \,\,\, \forall\,  \ell \in I_1 \right\} \\
& = {{A}_x \left(kt,{\frac{r}{32 \sqrt{mn}}},{L}, {U^c} \cap Q_t\right). }              \end{aligned} $$
Thus, the claim in this case follows from Theorem \ref{main cor} applied with {$S$ replaced with $U^c \cap Q_t$}, $N$ replaced by $L$, and $t$ replaced with $kt$.\\ Also if $\{1,\dots,L\}=J_{1}$, we have $  \sum_{i=1}^\ell |I_i|=0, d_{J,L}=0$, and $d'_{J,L}=0$. Moreover, since ${\frac{r}{32 \sqrt{mn}} \underset{\equ{ineq beta3}}{\le} 1}$, we get
$$
\begin{aligned}
\left\{ s \in M_{m,n}^{{r}/{32 \sqrt{mn} }}: g_{\ell kt}h_sx \in  {Q_t}^c \,\,\, \forall\,  \ell \in J_1 \right\}  
&  \subset A_x\left(kt,1,L,Q_t^c\right)\\
& \underset{\equ{qt}}{=}  A_x\left(kt,1,L,X_{>  C_{\alpha}^3 e^{mnt}}^t\right)
.  
\end{aligned}
$$
Hence, the claim in this case follows from Corollary \ref{fin cor} applied with {$M=C_{\alpha}^3 e^{mnt}$} and $N$ replaced with $L$.\bigskip \\}
In the inductive step, let $L'>L$ be the next integer for which an equation similar to \equ{induc eq} is satisfied. We have two cases. Either
\eq{case1}{\{1,\dots, L'\}=\{1,\dots,L\} \sqcup I_{\ell'+1}}
or
\eq{case2}{\{1,\dots, L'\}=\{1,\dots,L\} \sqcup J_{\ell+1}.}
We start with the case \equ{case1}. Note that in this case we have
\eq{d1}{d_{J,L'}=d_{J,L}+1 \text{ and } d'_{J,L'}=d'_{J,L}.}
Also, it is easy to see that every cube of {side-length} ${\theta}{e}^{-(m+n)Lkt}$ in $M_{m,n}$ can be covered with at most $2^{mn} {e}^{mn(m+n)kt}$ cubes of {side-length} $\theta {e}^{-(m+n)(L+1)kt}$.
Therefore, by using the induction hypothesis and in view of \equ{L case}, we can cover $Z(J)$ with at most 
 \eq{ind ine}
{
2^{mn} \left(\frac{C_\alpha^2 }{\theta^{mn}} \right)^{d'_{J,L}+1}\left((2^9 {mn})^{mn} K_0 \right) ^{d_{J,L}+1}{e}^{mn(m+n)(L+1)kt}
 \cdot 
D_1^{\sum_{i=1}^{\ell'} |I_i|-d_{J,L}}
D_2^ {\sum_{j=1}^{\ell } |J_j|}
}
\ignore{\le D_1^{\sum_{i=1}^\ell |I_i|-d_{J,L}} \cdot D_2^ {\sum_{j=1}^{\ell '} |J_j|} \cdot K_0^{d_{J,L'}+1} \cdot {C_5}^{d_{J,L'}} \cdot e^{(L+1)mn(m+n)t } }
cubes of {side-length}  $ {\theta}{e}^{-(m+n)(L+1)kt}$.
Now let $B$ be one of the cubes of {side-length} $ \theta e^{-(m+n)(L+1)kt}$ in the aforementioned  cover such that $B \cap Z(J) \neq \varnothing$. 
Clearly
\eq{intnew}{
B\text{ can be covered by }
\left( \frac{2\theta}{\frac{r}{32 {mn}}}\right)^{mn}\text{ cubes }
\text{of {side-length} }\frac{r{e}^{-(m+n)(L+1)kt}}{32 {mn}}.
}
Let $B_r$ be one of such cubes that has non-empty intersection with $Z(J)$, and let $s \in B_r \cap Z(J)$. Since $s \in Z(J)$, it follows that $g_{(L+1)kt}h_sx \in  U^c \cap Q_{t}$. Therefore, if we denote the center of $B_r$ by $s_0$, we have
\eq{h0}{{g_{(L+1)kt}h_{s_0}x \in B^H \left(\frac{r}{32 \sqrt{mn} }\right)(U^c \cap Q_{ t}) \subset \partial_{r} (U^c \cap Q_{ t}).}} Moreover, for any $s' \in B_r$ and any positive integer $1 \le i \le L'-(L+1)$ we have:
\eq{induc main1}{
\begin{aligned}
g_{(L+1+i)kt}h_{s'}x
&=g_{ikt}(g_{(L+1)kt}h_{s'-s_0}g_{-(L+1)kt})(g_{(L+1)kt}h_{s_0}x)\\
& =g_{ikt} h_{e^{(m+n)(L+1)kt}(s'-s_0)}(g_{(L+1)kt}h_{s_0}x)
\end{aligned}}
It is easy to see that the map $s' \rightarrow e^{(m+n)(L+1)kt}(s'-s_0)$ maps $B_r$ into {${{B\left(\frac{r}{32 \sqrt{mn}} \right)}}$.} Hence, by \equ{induc main1}
$$
\begin{aligned}
& \left \{s' \in B_r: g_{(L+1+i)kt}h_{s'}x \in U^c \cap Q_t \,\,\, \forall \, i \in \{1,\cdots, L'-(L+1) \} \right \}  \\
& \subset  e^{-(m+n)(L+1)kt} A_{g_{(L+1)kt}h_{s_0}x}\left(kt,\frac{r}{32 \sqrt{mn}}, L'-(L+1),U^c \cap Q_t \right) +s_0.
\end{aligned}
$$ \smallskip
So, in view of the above inclusion and \equ{h0}, we can go through the same procedure and apply Theorem \ref{main cor} with $t$ replaced with $kt$, {$S$ replaced with $U^c \cap Q_t$}, $N$ replaced with $|I_{\ell'+1}|-1=L'-(L+1)$, and $x$ replaced with $g_{(L+1)kt}h_{s_0}x$, and conclude that $B_r \cap Z(J)$ can be covered with at most $$\left(\frac{4r}{\theta}\right)^{mn}K_0 e^{mn(m+n)(|I_{\ell' +1}|-1)kt}D_1^{|I_{\ell'+1}|-1}$$ cubes of {side-length} ${\theta}{e}^{-(m+n)L'kt}$.
Therefore, in view of \equ{intnew}, the set $B \cap Z(J)$ can be covered with at most 
\begin{align*}
& 2^{mn}\left( \frac{\theta}{\frac{r}{32 {mn}}}\right)^{mn}  \left(\frac{4r}{\theta} \right)^{mn}K_0 {e}^{mn(m+n)(|I_{{\ell'} +1}|-1)kt}D_1^{|I_{{\ell'}+1}|-1}\\
& =K_0 \left({2^8 {mn}}\right)^{mn} {e}^{mn(m+n)(|I_{{\ell'} +1}|-1)kt}D_1^{|I_{{\ell'}+1}|-1}
\end{align*}
cubes of {side-length} ${\theta}{e}^{-(m+n)L'kt}$.
This, combined with \equ{ind ine} which is an upper bound for the number of cubes of {side-length} ${\theta}{e}^{-(m+n)(L+1)kt}$ in $M_{m,n}$ needed to cover $Z(J)$, implies that $Z(J)$ can be covered with at most  
\begin{align*}
& \left(K_0 \left({2^8 {mn}}\right)^{mn} {e}^{mn(m+n)(|I_{{\ell'} +1}|-1)kt}D_1^{|I_{{\ell'}+1}|-1} \right) \cdot   \\ & 2^{mn} \left(\frac{C_\alpha^2 {}}{\theta^{mn}}\right)^{d'_{J,L}+1} \left((2^9 {mn})^{mn} K_0 \right) ^{d_{J,L}+1}{e}^{mn(m+n)(L+1)kt}  
D_1^{\sum_{i=1}^{\ell'} |I_i|-d_{J,L'}} 
D_2^ {\sum_{j=1}^{\ell } |J_j|} \\
& \underset{\equ{d1}}{=} \left(\frac{C_\alpha^2  {}}{\theta^{mn}}\right)^{d'_{J,L'}+1}\left((2^9 {mn})^{mn} K_0 \right) ^{d_{J,L'}+1}{e}^{mn(m+n)Lkt}
D_1^{\sum_{i=1}^{\ell'} |I_i|-d_{J,L'}}
D_2^ {\sum_{j=1}^{\ell } |J_j|} 
\end{align*}
cubes of {side-length} ${\theta}{e}^{-(m+n)L'kt}$.
This ends the proof of the claim in this case. 
\smallskip

Next assume that \equ{case2} holds. Note that in this case
\eq{d2}{d_{J,L'}=d_{J,L}\text{ and } d'_{J,L'}=d'_{J,L}+1.}
Take a covering of $Z(J)$ with cubes of {side-length} $\theta e^{-(m+n)kLt}$ in $M_{m,n}$, suppose $B'$ is one of the cubes in the cover such that $B' \cap Z(J) \neq \varnothing$, and let $s_1$ be the center of $B'$. Then, since $\sqrt{mn} \theta \le 1$, it is easy to see that:
{\eq{h1}{g_{Lkt}h_{s_1}x \in B^H(\sqrt{mn}\theta)(U^c \cap Q_{ t}) {\subset} \partial_{1} Q_{ t}.}}
On the other hand, for any $s \in B'$ and any positive integer $1 \le i \le L'-L$ we have:
\eq{induc main}{
\begin{aligned}
g_{(L+i)kt}h_sx 
&=g_{ikt}(g_{Lkt}h_{s-s_1} g_{-Lkt})(g_{Lkt}h_{s_1}x) \\
& =g_{ikt}h_{e^{(m+n)Lkt}(s-s_1)}(g_{Lkt}h_{s_1}x).
\end{aligned}}
Note that that the map $s \rightarrow e^{(m+n)Lkt}(s-s_1)$ maps $B'$ into {${{B(1)}}$.} Thus, by \equ{induc main}

$$
\begin{aligned}
& \left \{s \in B': g_{(L+i)kt}h_{s}x \in Q_t^c \,\,\, \forall \, i \in \{1,\cdots, L'-L \} \right \}  \\
& \subset  e^{-(m+n)Lkt} A_{g_{Lkt}h_{s_1}x}\left(kt,1, L'-L,Q_t^c \right) + s_1\\
& =  e^{-(m+n)Lkt} A_{g_{Lkt}h_{s_1}x}\left(kt,1, L'-L,X_{>  C_{\alpha}^3 e^{mnt}} \right) + s_1
\end{aligned}
$$
So in view of the above inclusion and \equ{h1}, we can apply Corollary \ref{fin cor} with \linebreak $M=C_\alpha^3 e^{mnt}$, $g_{Lkt}h_{s_1}x$ in place of $x$, and $|J_{\ell+1}|=L'-L$ in place of $N$. This way, we get that the set $B' \cap Z(J) $ can be covered with 
at most
$$\begin{aligned}
& \frac{C_\alpha {}}{\theta^{mn}}  D_2^{|J_{\ell+1}|} \cdot {e}^{mn(m+n)k(|J_{\ell+1}|)t} \cdot  \frac{\max\big(\tilde{\alpha}^{{t}}(g_{Lkt}h_{s_1}x),1\big)}{M} \\
& \underset{\equ{p1},\, \equ{h1}}{\le} \frac{C_\alpha^2 {}}{\theta^{mn}}  D_2^{|J_{\ell+1}|} \cdot {e}^{mn(m+n)k(|J_{\ell+1}|)t}
\end{aligned}$$ 
cubes of {side-length} ${\theta}{e}^{-(m+n)kL't}$.
From this, combined with the induction hypothesis, we conclude that $Z(J)$ can be covered with at most 
$$
\begin{aligned}
& \left(\frac{C_\alpha^2 {}}{\theta^{mn}} D_2^{|J_{\ell+1}|} \cdot {e}^{mn(m+n)(|J_{\ell+1}|)kt} \right) \cdot   \left(\frac{C_\alpha^2 {}}{\theta^{mn}}\right)^{d'_{J,L}+1} \\
& \cdot \left((2^9 {mn})^{mn} K_0 \right) ^{d_{J,L'}+1} {e}^{mn(m+n)Lkt} D_1^{\sum_{i=1}^{\ell'} |I_i|-d_{J,L}} \cdot D_2^ {\sum_{j=1}^{\ell } |J_j|} \\
& \underset{\equ{d2}}{=}\left(\frac{C_\alpha^2 {}}{\theta^{mn}}\right)^{d'_{J,L'}+1}\left((2^9 {mn})^{mn} K_0 \right) ^{d_{J,L'}+1} {e}^{mn(m+n)L'kt} 
D_1^{\sum_{i=1}^{\ell'} |I_i|-d_{J,L}} 
D_2^ {\sum_{j=1}^{\ell +1} |J_j|}
\end{aligned}
$$
cubes of {side-length} ${\theta}{e}^{-(m+n)L'kt} $,
finishing the proof of the claim. \end{proof}

Now by letting $L=N$, we conclude that $Z(J)$ can be covered with at most \eq{last step}{
 \left(\frac{C_\alpha^2 {}}{\theta^{mn}}\right)^{d'_{J,N}+1}\left((2^9 {mn})^{mn} K_0 \right) ^{d_{J,L'}+1} {e}^{mn(m+n)Nkt} D_1^{|I|-d_{J,N}} 
 D_2^{|J|}  }
cubes of {side-length} ${\theta}{e}^{-(m+n)Nkt} $ in $M_{m,n}$. 

Clearly \eq{lastre}{d'_{J,N} \le d_{J,N}+1.} Also, note that since $d_{J,N} \le \max (|I|,|J|)$, the exponents $|I|-d_{J,N}, |J|-d_{J,N}$ in \equ{last step} are non-negative integers.  
So, in view of \equ{union} and \equ{last step}, the set
${A_x \left(kt,{\frac{r}{32 \sqrt{mn} }},N,U^c\right)}$ can be covered with at most:
{
\begin{align*}
& \sum_{  J \subset \{1,\dots,N\}}   \left(\frac{C_\alpha^2 {}}{\theta^{mn}}\right)^{d'_{J,N}+1}\left((2^9 {mn})^{mn} K_0 \right) ^{d_{J,N}+1} {e}^{mn(m+n)Nkt} D_1^{|I|-d_{J,N}} \cdot D_2^{|J|} \\
& \underset{\equ{lastre}}{\le} {e}^{mn(m+n)Nkt} \sum_{  J \subset \{1,\dots,N\}}   \left(\frac{C_\alpha^2 {}}{\theta^{mn}}\right)^{d_{J,N}+2}\left((2^9 {mn})^{mn} K_0 \right) ^{d_{J,N}+1} D_1^{|I|-d_{J,N}}  D_2^ {|J|}\\
 & \le \frac{C_0}{\theta^{2mn}} {e}^{mn(m+n)Nkt} \sum_{ J \subset \{1,\dots,N\}}   D_1^{|I|-d_{J,N}}  D_2^ {|J|} \cdot \left(\frac{C_0}{\theta^{mn}} \right)^{d_{J,N}} \\
 & =\frac{C_0}{\theta^{2mn}} {e}^{mn(m+n)Nkt} \sum_{ J \subset \{1,\dots,N\}}   D_1^{N-|J|-d_{J,N}}  D_2^ {|J|-d_{J,N}} \cdot \left( \sqrt{\frac{C_0  D_2}{\theta^{mn}} }\right) ^{2 d_{J,N}}\end{align*}
cubes of {side-length} ${\theta}{e}^{-(m+n)Nkt}$ in $M_{m,n}$,  where $C_0:={C_\alpha^4(2^9 {mn})^{mn} K_0  }\ge 1$.

\ignore{
 \\ 
&  \stackrel{\small{(1)}}{\le} \frac{C_0}{\theta^{2mn}} {e}^{mn(m+n)Nkt} \cdot \left( D_1+D_2+ \sqrt{\frac{C_0 \cdot D_2}{\theta^{mn}} } \right)^N\\
& \underset{\equ{C1},\, \equ{C2}}{=} \frac{C_0}{\theta^{2mn}} {e}^{mn(m+n)Nkt}  \left( 1-  K_1 \mu ({{{\sigma _{2 \sqrt{mn}{\theta}}}{ U}}})+\frac{K_2{e}^{-\lambda kt}}{r^{mn}}+(k-1)C_{1}^k  {e^{-\frac{t}{2}}  }   + \sqrt{\frac{(k-1)C_0C_{1}^{k}}{\theta^{mn}}}e^{-\frac{t}{4}} \right)^N \\
& \le \frac{C_0}{\theta^{2mn}} {e}^{mn(m+n)Nkt} \left(1-  K_1 \mu ({{{\sigma _{2 \sqrt{mn}{\theta}}}{ U}}})+\frac{K_2{e}^{-\lambda kt}}{r^{mn}} + {\frac{k-1}{\theta^{mn}}}C_{3}^k  {e}^{-\frac{t}{4}} \right)^N 

 where $C_0={C_\alpha^4(2^9 {mn})^{mn} K_0  }$ and $C_{3}=2C_{1} \cdot \max (1,C_0)$
 and inequality $(1)$ above results from the following lemma.}
 }
 \smallskip
 
To simplify the last expression we will use an auxilliary 

 \begin{lem} For any $n_1,n_2,n_3 > 0$ it holds that
 $$\sum_{  J \subset \{1,\dots,N\}}   n_1^{N-|J|-d_{J,N}}  n_2^ {|J|-d_{J,N}} n_3^{2 d_{J,N}} 
 \le   \left( n_1+n_2+ n_3 \right)^N.$$
 \end{lem}
 \begin{proof}
 Define the map  $\phi: \{1,\dots, N\}  \rightarrow \{n_1,n_2,n_3\}^N$ by
 $$\phi(J) =(x_1,\dots,x_N)     , $$
 where for any $i \in \{1,\dots,N \}$, $x_i$ is defined as follows:
 $$ x_i:=\begin{cases} n_1 & \text{if} \,\, i\in I \,\, \text{and} \,\, \left(i-1 \in I \,\, \text{or} \,\, i=1 \right);
 \\ n_2 & \text{if} \,\, i\in J \,\, \text{and} \,\, \left(i+1 \in J \,\, \text{or} \,\, i=N \right); \\
 n_3 & \text{otherwise}.
 \end{cases}   $$
It is easy to see that $\phi$ is one to one; moreover for any $ J \subset \{1,\dots,N\},$ the number of $i \in \{1,\dots,N\}$ such that $x_i=n_1$ is $|I|-d_{J,N}=N-|J|-d_{J,N}$, and the number of $i \in \{1,\dots,N\}$ such that $x_i=n_2$ is $|J|-d_{J,N}$. Therefore for any $ J \subset \{1,\dots,N\}$, $\phi(J)$ corresponds to one of the terms of the form $n_1^{N-|J|-d_{J,N}} n_2^{|J|-d_{J,N}}n_3^{2d_{J,N}}$ in the multinomial expansion of $(n_1+n_2+n_3)^N$. Since $\phi$ is {injective} and there exists a one to one correspondence between $\{n_1,n_2, n_3\}^N$ and the terms in the expansion of $(n_1+n_2+n_3)^N$, we conclude that
$$\sum_{J \subset \{1,\dots,N\}}n_1^{N-|J|-d_{J,N}} n_2^{|J|-d_{J,N}}n_3^{2d_{J,N}} \le (n_1+n_2+n_3)^N$$
and the proof is finished.
 \end{proof}
 
 Applying the above lemma with $n_1=D_1$, $n_2=D_2,$ and $n_3=\sqrt{\frac{C_0  D_2}{\theta^{mn}} }$, we conclude that 
 ${A_x \left(kt,{\frac{r}{32 \sqrt{mn} }},N,U^c\right)}$ can be covered with at most
 \begin{align*}
 &\frac{C_0}{\theta^{2mn}} {e}^{mn(m+n)Nkt}   \left( D_1+D_2+ \sqrt{\frac{C_0 D_2}{\theta^{mn}} } \right)^N\\
 \underset{\equ{C1},\, \equ{C2}}{=} & \ \frac{C_0}{\theta^{2mn}} {e}^{mn(m+n)Nkt}  \left( 1-  K_1 \mu ({{{\sigma _{2 \sqrt{mn}{\theta}}}{ U}}})+\frac{K_2{e}^{-\lambda kt}}{r^{mn}}+(k-1)C_{1}^k  {e^{-\frac{t}{2}}  }   + \sqrt{\frac{(k-1)C_0C_{1}^{k}}{\theta^{mn}}}e^{-\frac{t}{4}} \right)^N \\
 \le &\ \frac{C_0}{\theta^{2mn}} {e}^{mn(m+n)Nkt} \left(1-  K_1 \mu ({{{\sigma _{2 \sqrt{mn}{\theta}}}{ U}}})+\frac{K_2{e}^{-\lambda kt}}{r^{mn}} + {\frac{k-1}{\theta^{mn}}}C_{3}^k  {e}^{-\frac{t}{4}} \right)^N \end{align*}
 cubes of {side-length} ${\theta}{e}^{-(m+n)Nkt}$ in $M_{m,n}$,  where   $C_{3}:=2C_{1}C_0$.
The  proof {of {Proposition} \ref{first1}}  is now complete.
\end{proof}

\section{
{An intermediate dimension bound}}\label{intermediate}
{{Recall that we are given $a > 0$ and a non-empty open $U\subset X$, and our goal is to estimate the \hd of $E({F^+_a},U) $ from above. The following technical theorem shows how to express $E({F^+_a},U) $ as the union of two sets, taking into account the behavior {of} trajectories with respect to the family $\{Q_t\}$ constructed in the previous section,  and estimate their dimension separately.}
} 

\begin{thm}\label{first}
{Let $\{Q_t\}_{{t > 0}}$ of $X$ be as in  \equ{qt}. Then:}
\begin{enumerate}
\item
There exists {$C_{1} \ge 1$} such that for all {${t > 2}$} 
and for all $2 \le k \in \N$, the set
{\eq{S1}{{S(k,t,x) :=  \left\{h \in H :{{g_{Nkt}}}hx \in  Q_{t}^c \,\,\, \forall N \in \N    \right\}}}}
{satisfies
{\eq{S1 codim}{{\codim {S(k,t,x)}} \ge {\frac{1}{(m+n)k  } 
\left( \frac{1}{2}- \frac{\log \big((k-1)C_{1}^{k}\big)}{t} 
\right)}.}}}
\item
There exist
 $p \ge m+n$, {$0<r_{2}<  \frac{{1}}{16 \sqrt{mn}}$}, {$0 < C_{2} \le 1$} and 
 $b_0,b,
 K_1,K_2, C_{3}, \lambda   > 0
  $ such that  {for all {$ t \in a\N$}}
  {with ${t > 2}$},
all $2 \le k \in \N$, all $r$ satisfying
{\eq{ineq beta2}{{ {e^{ \frac{b_0 - kt}{b}}
} \le r \le \min( C_{2}{e}^{-p  t}, {r_{2}})},}} all {$\theta\in\left[4r ,\frac {1}{2\sqrt{mn}}\right]$}, all $x \in X$, and for all open subsets $U$ of $X$
we have 
\eq{S2 codim}{ 
{{\codim\big({\{h\in H\ssm S(k,t,x) : hx\in E({F^+_a},U) \}}\big)
} 
\ge {\frac{  K_1\mu \big({{{\sigma _{2 \sqrt{mn}{\theta}}}{ U}}}\big)-\frac{K_2 {e}^{-\lambda kt}}{r^{mn}}- {\frac{k-1}{\theta^{mn}}}C_{3}^{k}  
{e}^{- {t}/{4}} } {kt (m+n)}}}
.}
\end{enumerate}
 \end{thm}

Informally speaking, $S(k,t,x)$ is the set of $h\in H$ such that along some arithmetic sequence (of times which are multiples of $kt$) the orbit of $hx$ visits complements of large compact subsets of $G$.  {The dimension of $S(k,t,x)$ and the dimension of the set ${\{h\in H\ssm S(k,t,x) : hx\in E({F^+_a},U) \}}$ are estimated separately.}

\begin{proof}[Proof of Theorem \ref{first}]
 Take $\{Q_t\}_{{t > 0}}$ as in  \equ{qt},
and let $U$ be an open subset of $X$. 

\medskip
\noindent\textbf{Proof of (1):}
{Let {${t >2}$,} 
and take $2 \le k \in \N$ {and $x\in X$}.  Our goal is to find an upper bound for the Hausdorff dimension of the set {$S(k,t,x) $ defined in  \equ{S1}; {equivalently,}
$$\dim  S(k,t,x) =\dim \left \{s \in M_{m,n}: {{g_{Nkt}}}h_sx \in  Q_{t}^c \,\,\, \forall N \in \N       \right \} .               $$
In view of the countable stability of Hausdorff dimension  it suffices to 
estimate the  dimension of 
$$\left \{s \in {B(1)}: {{g_{Nkt}}}h_sx \in  Q_{t}^c \,\,\, \forall N \in \N       \right \} 
,$$
which, due to \equ{qt}, coincides with $\bigcap_{N \in \N} {A_x\left(kt,1,N,X_{>  C_{\alpha}^3 e^{mnt}}^t\right)}$.}

Applying Corollary \ref{fin cor} with $M=C_\alpha^3e^{2mnt}$, we get for any $x \in X$ and for any {$0<\theta \le \frac{1}{\sqrt{mn}}$}:
{ 
$${\begin{aligned}
    & \dim   \bigcap_{N \in \N}{A_x\left(kt,1,N,X_{>   C_{\alpha}^3 e^{mnt}}^t\right)}  \\
 & \le \lim_{N \rightarrow \infty}  \frac{ \log  \frac{C_\alpha {}}{\theta^{mn}}(k-1)^N  C_{1}^{kN} {e}^{(mn(m+n)Nk -\frac{N}{2}){t}} \cdot   \frac{{ \max(\tilde{\alpha}^{{t}}(x),1) }}{ M} }{- \log  {\theta}{e}^{-(m+n)Nk{t}}} \\
 & = \frac{\log (k-1)C_{1}^ke^{(mn(m+n)k -\frac{1}{2}){t}}}{kt(m+n)} \\
 & = {mn - { \frac{1}{(m+n) } \cdot \left(\frac{1}{2k}-\frac{\log (k-1)}{kt}- \frac{\log C_{1}}{t} \right)}}\\
 & =  {mn- { \frac{1}{k(m+n) } \cdot \left(\frac{1}{2}-\frac{\log \left(C_{1}^k(k-1)\right)}{t} \right)}, }
\end{aligned}}$$}
{where $C_{1}$ is as in Corollary \ref{fin cor}.}}
\medskip

\noindent \textbf{Proof of (2):} { Let {$a >0$}, $2 \le k \in \N$, $x \in X$, and let $t = \ell a$ for some $\ell\in\N$}. Our goal is to find an upper bound for the Hausdorff dimension of the set 
$$
\{h\in H\ssm S(k,t,x) : hx\in E({F^+_a},U) \} .
$$
{Recall that
$${S(k,t,x)^c =  \left\{h \in H:{{g_{Nkt}}}hx \in  Q_{t} \text{ for some } N \in \N    \right\}}.
$$
Therefore}
{
$${\begin{aligned}
& \{h\in H\ssm S(k,t,x) : hx\in E({F^+_a},U) \}  \\  =& \left \{h \in H: hx \in  E({F^+_a},U) \bigcap \left(\bigcup_{N \in \N}  g_{-Nkt}    Q_{t} \right)\right \}   \\ \subset & \left \{ h\in H: hx \in \bigcup_{N \in \N}  g_{-Nkt}   \big(Q_{t} \cap E({F^+_a},U)  \big) \right \}. \end{aligned}}$$}
Now suppose {that $t\ge2$},
and let $N \in \N$ and $r>0$ be such that \equ{ineq beta2} is satisfied, where $b_0,b,C_{2},{r_{2}}$ are as in Lemma \ref{first1}. Similar to the proof of part $(1)$ and in view of countable stability of Hausdorff dimension it suffices to find an upper bound for the dimension of the set 
$$E'_{N,x,r}:=  \left\{s \in {B\left( {{\frac{re^{-(m+n)Nkt}}{32  \sqrt{mn} }} }\right)} : h_sx \in g_{-Nkt}  \big(Q_{t} \cap E( {F^+_a},U)  \big) \right\}$$
for any $x \in X$. Now let $x \in X$ {and $s \in E'_{N,x,r}$. Then
$$
\begin{aligned}
g_{ikt} g_{Nkt}h_sx 
& = g_{ikt} (g_{Nkt}h_sg_{-Nkt})g_{Nkt}x \\
& =   g_{ikt} h_{e^{(m+n)Nkt}s} (g_{Nkt}x) \in U^c                  \quad\forall\,i\in\N,
\end{aligned}
$$
and at the same time $e^{(m+n)Nkt}s \in {B\left( {{\frac{r}{32  \sqrt{mn} }}}\right)}.$
It follows that 
 \eq{fininc}{E'_{N,x,r} \subset e^{-(m+n)Nkt} \left( \bigcap_{i \in \N} A_{g_{Nkt}x} \left( kt,{\frac{r}{32  \sqrt{mn} }}, i,U^c \right) \right) .}}
It is easy to see that if $E'_{N,x,r}$ is non-empty, then $g_{Nkt}x \in \partial_{\frac{r}{32 \sqrt{mn}}} \left(Q_{t} \cap U^c \right)$.
Now take $K_1$, $K_2$, $C_0$, $C_{3}$, $\lambda$ as in Lemma \ref{first1}. By Lemma \ref{first1} applied to $x$ replaced with {$g_{Nkt}x$, 
{and using the fact that {the \hd\ is preserved by homotheties},}
we have for any $\theta\in\left[4r ,\frac {1}{2\sqrt{mn}}\right]$:
$${\begin{aligned} 
 \dim E'_{N,x,r}  & \underset{\equ{fininc}}{\le} \dim \left( e^{-(m+n)Nkt}\left( \bigcap_{i \in \N} A_{g_{Nkt}x} \left( kt,{\frac{r}{32  \sqrt{mn} }}, i,U^c \right) \right)  \right)\\
& 
{=}  \dim \bigcap_{i \in \N} A_{g_{Nkt}x} \left( kt,{\frac{r}{32  \sqrt{mn} }}, i,U^c \right)  \\
& \le \lim_{i \rightarrow \infty} \frac {\log \left(\frac{C_0}{\theta^{2mn}}{{e}^{mn(m+n)Nkt}   \left(1-  K_1 \mu \big({{{\sigma _{2 \sqrt{mn}{\theta}}}{U}}}\big)+\frac{K_2{e}^{-\lambda kt}}{r^{mn}} + {\frac{k-1}{\theta^{mn}}}C_{3}^k  {e}^{-\frac{ t}{4}} \right)^i} \right)}{{- \log {\theta}{e}^{-(m+n)ikt}}} \\
& \le mn - \frac{-\log \left(1-  K_1\mu \big({{{\sigma _{2 \sqrt{mn}{\theta}}}{( U)}}}\big)+\frac{K_2{e}^{-\lambda kt}}{r^{mn}}+{\frac{k-1}{\theta^{mn}}}C_{3}^{k}  {e}^{-\frac{t}{4}} \right)}{ (m+n)kt}\\
& \le {mn - \frac{  K_1\mu \big({{{\sigma _{2 \sqrt{mn}{\theta}}}{( U)}}}\big)-\frac{K_2{e}^{-\lambda kt}}{r^{mn}}-{\frac{k-1}{\theta^{mn}}}C_{3}^{k}  {e}^{-\frac{t}{4}} }{(m+n)kt}}. 
\end{aligned}}$$}
  This finishes the proof.
\end{proof}

\section{{Theorem \ref{first} $\Rightarrow$ Theorem \ref{dimension drop 3} $\Rightarrow$ 
Theorem \ref{dimension drop 2} $\Rightarrow$ applications}}\label{easyproofs}
  {We begin with a remark that $${\widetilde E({F^+_a},U) = \bigcup_{{j}\in N}g_{-a{j}}E({F^+_a},U),}$$
hence if an upper estimate for $\dim E({F^+_a},U)$ is proved, the same estimate holds for $\widetilde E({F^+_a},U)$ because of the countable stability of Hausdorff dimension and its invariance under diffeomorphisms. The same argument applies to  
$${\{h\in H: hx\in {\widetilde E({F^+_a},U)}\} = \bigcup_{{j}\in N}g_{-a{j}}\{h\in H: hg_{a{j}}x\in {\widetilde E({F^+_a},U)}\}g_{a{j}}}.$$
Therefore it is enough to prove Theorems \ref{dimension drop 2} and \ref{dimension drop 3} with $E({F^+_a},U)$ in place  of $\widetilde E({F^+_a},U)$.}

\smallskip

We now show how the two  {parts of Theorem \ref{first}} are put together.

\begin{proof}[Proof of Theorem \ref{dimension drop 3}]{Let $x \in X$ and $a>0$.}
{Recall that we are given the constants $p,{{r_{2}}},b, K_1,K_2, C_1,C_2, C_{2}, \lambda$ and a family of compact sets $\{Q_t\}_{{t >0}}$ such that statements (1) and (2) of Theorem \ref{first} hold.
To apply the theorem we need to choose $k\in\N$ and {$t \in a\N$}. Here is how to do it. First  define
\eq{k}{k:= \left \lceil{\max \left( \frac{4p}{m+n},\frac{2p(mn+2)}{\lambda},4bp \right)}\right \rceil}
(note that $k\ge 4$ since $p \ge m+n$), and then choose   {$t_1:=\max \left(K_1, 4 \log \big((k-1)C_{1}^{k}\big)\right)$}.
{We remark that $t_1 \ge 4 \log 3 > 4$, since $C_{1} \ge 1$ and $k\ge 4$.}
{Statement (1) of Theorem \ref{first} readily implies that {\eq{estimate1}{{\codim S(k,t,x)} \ge  \frac{1}{4k(m+n)  }  
}}
 {whenever $t\ge t_1$.
 }} 
 Now let  {\eq{constants}{{c := 
 C_{2},
 }}}
{\eq{r12}{\begin{aligned}
 r_3:=\min\Big({{c}}^2 e^{-b_0/b} ,\,{{c}}^{mn+2} \frac{K_1}{8K_2}, \qquad\qquad\qquad\qquad\qquad\qquad\qquad\\
{{c}}^3  \left(\frac{K_1}{8 (k-1)C_{3}^k} \right)^{24p},\,  {{{c}} e^{-2pt_1}}, \big(\frac{1}{2 \sqrt{mn}}\big)^{24pmn}, {r_{2}}
\Big) \end{aligned}}
\eq{r1first}{r_1:= r_3^{\frac{1}{24pmn}},}
and set 
{\eq{t define}{r:={r(U,a)}^ {24pmn} \quad\text{and}\quad t:= a\left \lceil {{\frac{1}{2ap}             \log \frac{{{c}}}{{r} }  }}\right \rceil   ,}}}
where 
${r(U,a)}$ is defined by \equ{cu}. {Note that in view of \equ{cu}, \equ{r1first}  and \equ{t define} {one has}
\eq{r1sec}{r \le r_3.}}
{Also,} it follows from 
\equ{t define} that 
\eq{t redefine}{{{{c}} e^{-2pt} \le r \le {{c}} e^{-2p(t-a)}.}}
{ Moreover,
{$$
{t \underset{\equ{t define}}{\ge} {\frac{1}{2p}             \log \frac{{{c}}}{{r} }  } \underset{\equ{r1sec}}{\ge} \frac{1}{2p}             \log \frac{{{c}}}{{r_3} } \underset{\equ{r12}}{\ge} 
t_1,}$$}
and
{\eq{t bound 2}{t \underset{\equ{t define}}{\ge } {\frac{1}{2p}             \log \frac{{{c}}}{{r} }  } \underset{\equ{cu},\, \equ{t define}}{\ge}{\frac{1}{2p}             \log \frac{{{c}}}{{ce^{-24a  pmn}} }} = 12amn
.}}
\ignore{, hence $t \ge t_1$. {Furthermore, since $t_1 \ge 4$ it follows from 
\equ{t redefine} that 
\eq{t reredefine}{\ r \le {{c}} e^{-apt}.}}}}
We now claim that the inequalities \equ{ineq beta2} are satisfied. Indeed, the second inequality  $r \le \min(  {C_{2}e^{-p  t},{r_{2}}})$ follows immediately because
\begin{itemize}
{\item $r\le r_3$ by  \equ{r1sec}, and $r_3 \le {{r_{2}}}$ by \equ{r12};}
\item $r \le C_{2} e^{-2p(t-a)}$ by \equ{constants} and \equ{t redefine}, and $t  \ge 4a$ {by \equ{t bound 2}}. 
\end{itemize} }

{Furthermore}, 
{we have
{
$$
{e^{ \frac{b_0 - kt}{b}}
 \underset{\equ{k}}\le 
e^{ \frac{b_0}{b} - 4pt}}
 \underset{\equ{t redefine}}\le \frac{e^{b_0/b}}{{{c}}^2} \cdot r^2 \underset{\equ{r1sec}}\le \frac{e^{b_0/b}r_3}{{{c}}^2}  \cdot r  \underset{\equ{r12}}\le  r,$$}
so the claim follows.}
{We therefore can apply 
\equ{S2 codim} to any {$\theta\in\left[4r ,\frac {1}{2\sqrt{mn}}\right]$}. We put {$\theta := \min (\theta_U, \frac{{1}}{2 \sqrt{mn}})$,} which is not greater than  {$\frac {1}{2\sqrt{mn}}$} by definition. To show that  it not less than ${4r}$, write 
{
$$\begin{aligned}
\theta  
& \underset{\equ{cu},\, \equ{r1sec}}\ge \min \left( r^{\frac{1}{24pmn}},\frac{{1}}{2 \sqrt{mn}}\right)  \underset{\equ{r12},\, \equ{r1sec}} =r^{\frac{1}{24pmn}} \\
& = \frac{r}{ r^{1 - \frac{1}{24pmn}}}\underset{\equ{r1sec}}\ge \frac{r}{ {r_3}^{1 - \frac{1}{24pmn}}}
 \ge \frac{r}{ {r_3}^{1/2}} \underset{\equ{r12}} \ge {4r.}
\end{aligned}$$}
Thus we can conclude that \eq{goodestimate}{{{\codim\big({\{h\in H\ssm S(k,t,x) : hx\in E({F^+_a},U) \}}\big)
} } \ge
{\frac{  K_1 {\mu \big({{{\sigma _{2 \sqrt{mn}{\theta}}}{ U}}}\big)}-\frac{K_2 e^{-\lambda kt}}{r^{mn}}- {\frac{k-1}{\theta^{mn}}}C_{3}^{k}  
e^{- {t}/{4}} } {kt (m+n){}}}.}
{Observe that since $\theta \le \theta_U$, {$\mu \big({{{\sigma _{2 \sqrt{mn}{\theta}}}{ U}}}\big)$}} is not less than $\mu(U)/2$ by  definition of $\theta_U$, see \equ{su1}. We now claim that the numerator in the right hand side of \equ{goodestimate} is not less than $K_1\mu(U)/4$. Indeed, we can write
{
\begin{align*}
{ \frac{k-1}{\theta^{mn}} C_{3}^k e^{-\frac{t}{4}}} & ={\frac{k-1}{\theta^{mn}} C_{3}^k (e^{-6pt})^{\frac{1}{24p}}}\underset{\equ{t redefine}}\le \frac{k-1}{\theta^{mn}} C_{3}^k \left(\frac{r^3}{{{c}}^3} \right)^{\frac{1}{24p}}  \\
&  \underset{\equ{constants}}= (k-1)C_{3}^k \left( \frac{r}{{{c}}^3} \right) ^{\frac{1}{24p}} \cdot \left(\frac{r^{\frac{1}{24pmn}}}{\theta}\right)^{mn} \cdot r^{\frac{1}{24p}}\\ & \underset{\equ{cu}, \, \equ{r12},\, \equ{t define}, \,  \equ{r1sec}}\le  (k-1)C_{3}^k \left( \frac{r_3}{{{c}}^3} \right) ^{\frac{1}{24p}}
 \cdot 1 \cdot \mu(U)  \underset{\equ{r12}}\le \frac{K_1}{8}\mu(U)
 \end{align*}}
and 
{
\begin{align*}
{\frac{K_2e^{-\lambda kt}}{r^{mn}}} & {\underset{\equ{k}}\le  \frac{K_2e^{-\lambda \cdot \frac{2p(mn+2)}{\lambda} t}}{r^{mn}}  = \frac{K_2(e^{-2p t})^{mn+2} 
}{r^{mn}}}
 \underset{\equ{t redefine}}\le  \frac{K_2 \big( \frac{r}{{{c}}}\big)^{mn+2}
}{r^{mn}}\\
&= K_2 \frac{r}{{{c}}^{mn+2}} \cdot r \underset{\equ{cu}, \,\equ{r12},\, \equ{t define}, \, \equ{r1sec}}\le K_2 \cdot \frac{r_3}{{{c}}^{mn+2}} \cdot {\mu(U)}^{24pmn} 
 \\ & \underset{\equ{r12}}\le K_2 \cdot \frac{K_1}{8K_2} \cdot {\mu(U)} ^{24pmn}\le \frac{K_1}{8} \mu(U).
\end{align*}}
Thus \equ{goodestimate} implies $${{\codim\big({\{h\in H\ssm S(k,t,x) : hx\in E({F^+_a},U) \}}\big)}
}  \ge
{\frac{ K_1\mu(U)} {4kt (m+n)}\,
{\underset{\equ{ineq beta2}} {\ge}}\,
{\frac{ K_1\mu(U)} {4k  (m+n)  \cdot \frac{1}{p}             \log  \frac{{{c}}}{r}}};} 
$$
hence, using \equ{estimate1}, we get 
$${ {\codim \big(\{h\in H: hx\in E({F^+_a},U)\} \big)} \ge  \frac{1}{4k(m+n)} {\min\Big(1,
{\frac{ pK_1\mu(U)} {          \log \frac{{{c}}}{r}}}\Big).}}$$
Finally, we claim that the minimum in the right hand side of {the above inequality}
is equal to {${\frac{ pK_1\mu(U)} {             \log \frac{{{c}}}{r}}}$.} 
Indeed, 
{$$
r \underset{\equ{ineq beta2}} \le {{c}} e^{-pt} \le {{c}} e^{-pt_1} < {{c}} e^{-pK_1}
\Longrightarrow\             \log \frac{{{c}}}{r} \ge   K_1p\ \Longrightarrow\   {\frac{ pK_1\mu(U)} {                \log  \frac{{{c}}}{r}}} < 1.$$}
Therefore
{\begin{align*}{\codim \big(\{h\in H: hx\in E({F^+_a},U)\} \big)} & \ge   \frac{pK_1\mu(U)}{4k(m+n)  \cdot \log  \frac{{{c}}}{r}} \\
& \underset{({{c= C_2}} \le 1)}\ge  \frac{pK_1}{4k(m+n)} \cdot \frac{\mu(U)}{ \log \frac{1}{r}} \\
& \underset{\equ{t define}}= \frac{K_1}{96kmn(m+n)} \cdot \frac{\mu(U)}{ \log \frac{1}{r(U,a)}}.
\end{align*}
}}
This finishes the proof.
\end{proof}

\ignore{For the proof of Theorem \ref{dimension drop 2} let us introduce some more notation. In what follows, if $P$ is a subgroup of $G$, we will denote by $B^{P}(r)$  the open ball of radius $r$ centered at the identity element with respect to the  
metric  on $P$ coming from the Riemannian structure induced from $G$. In particular, we have $$B^H(r) = \{h_s: s\in M_{m,n},\ \|s\|\le r\}.$$}

\begin{proof}[{Proof of Theorem \ref{dimension drop 2}
}]
{
{
{Denote by  $\tilde H$ the weak stable horospherical subgroup with respect to $F^+$ defined by
$$\tilde H:=\left\{ \begin{bmatrix}
s' & 0 \\
s & s''
\end{bmatrix}:s \in M_{n,m}, \ s' \in M_{m,m}, \ s'' \in M_{n,n}, \  \det(s')\det(s'')=1 \right \},     $$
Let $U$ be an open subset of $X$.
Choose $\eta>0$ sufficiently small so that for any $0<r< \eta$ the following conditions are satisfied:
\eq{rcond}
{\begin{split}
& \mu \big( \sigma_{r/2} U\big) \ge  \mu \big( U\big)/2, \\
& \theta_{\sigma_{r/2}U} \ge \frac{1}{2}\theta_U,
\end{split}}
where $\theta_U$ is as in \equ{su1}.
We choose $r'>0$ and $0<r<\eta$ sufficiently small such that the following properties are satisfied:
\begin{enumerate}
    \item 
 Every ${g \in B^{G}(r')} 
$  can be written as $g= h'h$, where
$h' \in {B^{\tilde H}}(r/4)$ and $h \in {B^{H}}(r/4)$.   
\item \begin{equation}\label{conjugate implied}
{g_tB^{\tilde H}(r)g_{-t} \subset B^{\tilde H}(2r) \text{ for any $0<r<\eta$ and }t\ge 0}\end{equation}
{(this can be done since for any {$t \ge 0$} the restriction of the map  {$g \to g_tgg_{-t}$} 
to   $\tilde H$
is non-expanding).}
\end{enumerate}
For $x\in X$ denote 
$${{{E_{x,r'}}:=}\, \big\{ g \in {B^{G}
}(r'):gx \in E(F^+_a,U) \big\}.}$$
{Clearly $E({F^+_a },U)$ can be covered by countably many sets {of type} $\{gx : g \in E_{x,r'}\}$.
Thus, in view of the countable stability of Hausdorff dimension, in order to prove the theorem it suffices to show} that for any  $x \in X$,
$${ \codim E_{x,r'} \gg \frac{\mu(U)}{\log \frac{1}{{r(U,a)}}}, }  $$
 where $r(U,a)$ is as in \equ{cu} and  {$c,r_1$ are as in Theorem \ref{dimension drop 3}.} 
 
Now let ${g \in B^{G}(r')} $  and suppose $g = h'h$, where
$h' \in {B^{\tilde H}}(r/4)$ and $h \in {B^{H}}(r/4)$, then
for any $y \in X$ {and any $t > 0$} we can write
$${
\begin{aligned}
{{\dist}}({g_t}gx,y) &\le {{\dist}}({g_t}h'hx,{g_t}hx) + {{\dist}}({g_t}hx,y)\\ &={{\dist}}\big(g_th'g_{-t}{g_t}hx,{g_t}hx\big)+ {{\dist}}({g_t}hx,y)\underset{\eqref{conjugate implied}}\le r/2 + {{\dist}}({g_t}hx,y).\end{aligned}}$$
Hence  
$g\in {E_{x,r'}}$ implies that {$h{x}$ belongs to $E({F^+_a },{\sigma
_{{r/2}}}U)$}, and by using
Wegmann's Product Theorem \cite{Weg} we conclude that: 
{\eq{wegmann}
{\begin{split}
\dim {E_{x,r'}} 
&\le {\dim} \left(\{ h \in {B^H}(r/4):hx \in E({F^+_a },{\sigma
_{r/2}}U)\} \times {B^{\tilde H}(r/4)}
\right)\\
& \le {\dim} \big(\{ h \in {B^H}(r/4):hx \in E({F^+_a },{\sigma
_{r/2}}U)\}\big) +  \dim  {\tilde H }\\
& \le {\dim} \big(\{ h \in H:hx \in E({F^+_a },{\sigma
_{r/2}}U)\}\big) +  \dim  {\tilde H }
\end{split}
}}
Note that by \equ{rcond} we have:
{
\eq{rprime}{
\begin{split}
{r(\sigma_{r/2} U,a)}& = \min \left( \mu(\sigma_{r/2}U),\theta_{\sigma_{r/2}U},{c}e^{-a},r_1 \right) \\ & \ge \min \left( { {\mu(U)}/{2}},\frac{\theta_{U}}{2},{c}e^{-a},r_1 \right) 
 \ge  \frac{1}{2}  \cdot r(U,a).
\end{split}}}
Therefore, {by Theorem \ref{dimension drop 3} applied $U$ replaced by $\sigma_{r/2}U$} and in view of \equ{rprime} and \equ{wegmann} {we get
{
$${\begin{split} 
 \codim E_{x,r'}  
& \ge { \codim \big(\{ h \in H:hx \in E({F^+_a },{\sigma
_{r/2}}U)\}\big) } \\
& \gg
\frac{\mu(\sigma_{r/2} U)}{\log \frac{1}{{r(\sigma_{r/2} U,a)}}}  \ge \frac{\frac{1}{2}\mu( U)}{\log  \frac{2}{{r( U,a)}}}  \underset{r(U,a) \, \le r_1 \, \le 1/2} \ge \frac{1}{4} \cdot \frac{\mu( U)}{\log  \frac{1}{{r( U,a)}}}\end{split}}$$}
} This ends the proof of the theorem.}}}
\end{proof}

{\begin{proof}[Proof of Corollary \ref{Cor2}]
Let $S$ be a $k$-dimensional smooth embedded submanifold of $X$, which we can assume to be compact. Then
it is easy to see that  one can find $\vre_0,\varkappa_1,\varkappa_2  >0$ 
such that 
$${\mu \big({\partial_\vre S}\big) \ge \varkappa_1 \vre^{\dim X -k}  }$$
and
$${\theta_{\partial_\vre S} \ge \varkappa_2 \vre
 }$$
 for any $0<\vre<\vre_0$.
Hence, in view of \equ{cu}, 
{$$r({\partial_\vre S,a} ) 
\ge \min \left ( r_1,{\varkappa_1}\vre^{\dim X -k},\varkappa_2\vre,ce^{-a} \right)
,$$
where $r_1,c$ are as in Theorem \ref{dimension drop 2}.
Therefore, if we denote $$\varkappa_0 := \min\left(\varkappa_1^{\dim X -k}, \varkappa_2 \right)\text{ and }p_0 = \max\big(\dim X -k,1\big),$$
we will have $r({\partial_\vre S,a} ) 
\ge \varkappa_0\vre^{p_0}$ as long as $\varkappa_0\vre^{p_0} < \min \left ( r_1,ce^{-a} \right)$.
By Theorem \ref{dimension drop 2} applied with $U= \partial_\vre S$ for  $\vre$ as above 
we have
$$
\codim E({F^+_a},\partial_{{\vre}} S)  
\gg \frac{{\mu (\partial_{{\vre}} S)}}{{\log \left(\frac{1}{{r}({\partial_{{\vre}} S,a} ) }\right)}} \ge  \frac{{\varkappa_2 {{\vre}}^{\dim X -k}}}{{\log \left(\frac{1}{{\varkappa_0{{\vre}}^{p_0}}  }\right)}} ,
$$
which implies \equ{forS} for a suitable choice of $\vre_S, c_S$, and $C_S$. The `in addition' part is proved along similar lines and is left to the reader.}
\end{proof}
\begin{proof}[Proof of Theorem \ref{di estimate}]
Recall that $X$ can be identified with the space of unimodular lattices in $\R^{m+n}$. 
It was essentially observed by Davenport and Schmidt in \cite{DS} (see also \cite{dani, KM2,  KW} for other instances of the so-called \emph{Dani Correspondence}) that {given $c< 1$, an $m\times n$  matrix $s$ is an element of $\mathbf{DI}_{m,n}(c)$ if and only if 
for large enough $t > 0$ the lattice $g_th_s\Z^{m+n}$ does not belong to a certain  subset $U_c$ of $X$ with non-empty interior.}
Indeed, 
the validity of  \equ{di} for large enough $N > 0$ is equivalent to a statement that for large enough $t$ there exists $\vv = \begin{pmatrix}-\p \\ \vq\end{pmatrix}\in\Z^{m+n}\nz$ such that the vector
$$g_th_s\vv = \begin{pmatrix}e^{nt}( s\vq - \p) \\ e^{-mt}\vq\end{pmatrix}$$
belongs to 
$$
\mathcal{R}_c := \left\{\begin{pmatrix}\x \\ \vy\end{pmatrix}\in\R^{m+n} : \|\x\| < c,\ \|\vy\| \le 1\right\}.
$$
 Now consider $$U_c := \big\{x\in X : x\cap \mathcal{R}_c = \{0\}\big\}.$$ 
{Take $\tau > 0 $ such that $e^{-n\tau} = \frac{1+c}2$. Then it is easy to see that a sufficiently small neighborhood of the lattice $g_{-\tau}\Z^{m+n}$ is contained in $U_c$; that is, the latter 
 has a non-empty interior. Thus $s\in\mathbf{DI}_{m,n}(c)$ is equivalent to $h_s\Z^{m+n}\in \widetilde E(F^+, U_c)$.} 
 An application of Theorem~\ref{dimension drop 3} 
 shows that the codimension of $\mathbf{DI}_{m,n}(c)$ in $\mr$ is positive.
 \end{proof}}
 {We refer the reader to \cite{BGMRV} for some recent results on the set of Dirichlet improvable vectors, and to \cite{KR} for an extension of the problem of improving Dirichlet's theorem to the set-up of arbitrary norms on $\R^{m+n}$.}

\section{Concluding remarks}

\subsection{
More precise estimates for   $\dim E(F^+,U)$}
Studying trajectories 
missing a given open subset has been a notable theme in ergodic theory. Such a  set-up is often referred to as `open dynamics' or `systems with holes', see e.g.\ \cite{FP, FS} and references therein. In particular, \cite[Theorem 1.2]{FP} considers a  conformal repeller supporting a Gibbs measure  and gives an asymptotic formula for the set of points missing a ball of radius $\varepsilon$, showing  the codimension to be asymptotically as $\varepsilon \to 0$ proportional to the measure of the ball. A similar formula was obtained by Hensley \cite{H}  in the setting of continued
fractions. See also \cite{DFSU} for a modern treatment of the subject.

In view of these results one can expect that in our set-up the codimension of $E(F^+,U)$ should also be asymptotically (as $\mu(U)\to 0$) proportional to the measure of $U$. In other words, conjecturally there should not be any logarithmic term in the right side of \equ{dbound}. However it is not clear how to improve our upper bound, as well as how to 
obtain a complimentary lower bound for $\dim E(F^+,U)$ using the exponential mixing of the action or any other method. The only known result supporting this conjecture  in a partially hyperbolic setting is a theorem of Simmons \cite{Si}  which establishes the asymptotics for the codimension of $E(F^+,U)$  in the set-up \equ{sln}--\equ{gt} and with $U$ being a complement of a large compact subset of $X$.

\ignore{The Hausdorff dimension of sets of points with non-dense orbits has been extensively studied in other types of dynamical systems as well. For example, Urbanski in \cite{U1} showed that if $M$ is a compact Riemannian manifold and $T : M \rightarrow M$ is a
transitive Anosov diffeomorphism then the Hausdorff dimension of the set of points with non-dense orbit under $T$ is full. He proved the same statement for Anosov
flows and expanding endomorphisms.  This work was improved in \cite{AN}, where a lower estimate for the set of points whose $T$-orbit stays away from a fixed open subset of $M$ was obtained. In \cite{Do}, Dolgopyat studied 
studied similar questions for Anosov flows and diffeomorphisms.
Also, in the setting of Teichm\"uller dynamics, Hausdorff dimension of sets of points with non-dense orbits has been studied in several papers (see  e.g.\ \cite{M} and \cite{MRC}).

\subsection{Precise estimates for the codimension of $E(F^+,U)$}
In view of results for a wide variety of dynamical systems (see  e.g,\ \cite{Si} and \cite{FP}) we expect that the codimension of $E(F^+,U)$ should be approximately some constant times the measure of $U$, and conjecturally there should not be any logarithmic term on the right side of \equ{dbound}. However it is not clear how to improve our upper bound, as well as how to 
obtain a complimentary lower bound for $\dim E(F^+,U)$, using the exponential mixing of the action.}

\subsection{Large deviations in homogeneous spaces}
Let $X = \ggm$ be an arbitrary finite volume \hs , let $\mu$ be  a $G$-invariant probability measure on $X$, and let $F^+=\{g_t\}_{t \ge 0}$ be a one-parameter subsemigroup of $G$   acting  ergodically on $(X,\mu)$. Given an open subset $U$ of $X$ and $0 < \delta \le 1$, let us say that a point $x \in X$ \emph{$\delta$-escapes $U$ on average with respect to
$F^+$} if $x$ belongs to
$$E_{\delta}(F^+,U):=\left\{x \in X: \lim \sup_{T \rightarrow \infty} \frac{1}{T} \int_0^T 1_{U^c}(g_tx)\,dt \ge \delta\right\}, $$
that is, to  the set of points in $X$ whose orbit spends at least $\delta$-proportion of time in $U^c$. Note that for any $0< \delta \le 1$ we have
\eq{lequ}{ E(F^+,U) \subset E_\delta(F^+,U),}
which means that the sets $E_\delta(F^+,U)$ are larger compared to $E(F^+,U)$; hence their dimension is greater than or equal to dimension of $E(F^+,U)$. Birkhoff's Ergodic theorem implies 
{
$$\lim\limits_{T\rightarrow\infty}\frac{1}{T}\int_0^T 1_{U^c}(g_tx)dt = \mu (U^c)\,\,\, \text {for almost all} \,\, x \in X .$$}
Hence, the set $E_\delta(F^+,U)$ has full measure for any $0< \delta \le \mu (U^c)$, and has zero measure whenever $\mu(U^c)<\delta \le 1 $. This motivates estimating the Hausdorff dimension of  $E_\delta(F^+,U)$ for $\mu(U^c)<\delta \le 1 $. 

Now let $F^+$ be $\Ad$-diagonalizable, and let $H$ be a subgroup of $G$ with the Effective Equidistribution Property (EEP) with respect to $F^+$. In a forthcoming work, by obtaining an explicit upper bound for $\dim E_\delta(F^+,U)$, we prove that for any non-empty open subset $U$ of a compact \hs \ $X$ there exists $\delta_U \in[\mu (U^c), 1]$ such that for any $\delta_U < \delta \le 1$ we have $\dim E_\delta (F^+,U)< \dim X$. This, in view of \equ{lequ}, will strengthen the main result of \cite{KM}  when $\Gamma$ is a uniform lattice. A similar result was proved in \cite{KKLM}  in the set-up \equ{sln}--\equ{gt} for trajectories divergent on average; {see also   \cite{MRC,RW} 
for extensions.}

\subsection{Dimension drop conjecture for arbitrary homogeneous spaces and arbitrary flows}
As we saw in this paper, 
height functions on the space of lattices provide a powerful tool for studying   orbits which spend a large proportion of time in the cusp neighborhoods. The construction of such functions for arbitrary homogeneous spaces was given by Eskin and Margulis  in \cite{EM}. This can be used to control geodesic excursions into cusps in any homogeneous space. For example, Guan and Shi in \cite{GS}  used a generalized version of the Eskin-Margulis function to extend the methods employed in \cite{KKLM} to arbitrary homogeneous spaces  and show that the set of points with divergent on average trajectories has less than full Hausdorff dimension. We believe that by taking a similar approach, and by combining the methods of this paper with {those of} \cite{EM} and \cite{GS}, one can potentially solve the Dimension Drop Conjecture for arbitrary homogeneous spaces and arbitrary flows. This project is a work in progress.

\ignore{Recall that
$${\theta_U=\min \left(\frac{1}{4 \sqrt{mn}},\sup \{ s > 0: \mu(\sigma_{3 \sqrt{mn}{\theta}}U) \ge \frac{1}{2} \mu(U)  \}\right)}$$
\ignore{and
\eq{k}{k:=\left \lceil{\max \left( \frac{2p+ \log C_4}{m+n},\frac{2p(mn+1)}{\lambda},4bp \right)}\right \rceil .}}
and
$${r(U,a)}=\min \left( \mu(U),{\theta_U}^{8mnp},r_1 \right),$$
where
$$ r_1=\min \left( r_3,{a}^{-pt_1}, \frac{K_1}{8K_2}, {a}^{-4pK_1}\right).$$
Note that $\theta_U > 4 \sqrt{mn}\cdot {r(U,a)}.$ Therefore, by \equ{S2 codim} applied with $\theta=\theta_U$ and $r={r(U,a)}$ and in view of  \equ{ineq beta1}, we conclude that for any $t>0$ and any integer $k \ge 2$ satisfying
\eq{ineq beta}{{ \max(C_4{a}^{-(m+n)kt},{a}^{- \frac{k}{b}(t-a)}) \le {r(U,a)} \le \min( C_{2}{a}^{-p  t}, r_3)},} we have
\eq{S2 codim1}{\codim S_2(U,k,t) \ge \frac{\log \left(1-  K_1\mu ({{{\sigma _{3 \sqrt{mn}{\theta}}}{( U)}}})+\frac{K_2 {a}^{-\lambda kt}}{{r(U,a)}^{mn}}+ \sqrt{\frac{k-1}{{\theta_U}^{mn}}}C_{3}^{k} t^{k/2} e^{-\frac{t}{4}} \right)}{{\log a} \cdot (m+n)kt}.}}



\bibliographystyle{alpha}

\end{document}